\documentclass[%draft,
11pt]{amsart}

\usepackage{amsmath}
\usepackage{amssymb}
\usepackage{latexsym}
\usepackage{amsfonts}
\usepackage{graphpap}
\usepackage[backrefs]{amsrefs}
\usepackage{bigdelim}
\usepackage{multirow}

\DeclareSymbolFont{SY}{U}{psy}{m}{n}
\DeclareMathSymbol{\emptyset}{\mathord}{SY}{'306}
\begin{document}
%Commands Used%

\newcommand{\ci}[1]{_{ {}_{\scriptstyle #1}}}
\newcommand{\ti}[1]{_{\scriptstyle \text{\rm #1}}}

\newcommand{\norm}[1]{\ensuremath{\|#1\|}}
\newcommand{\abs}[1]{\ensuremath{\vert#1\vert}}
\newcommand{\nm}{\,\rule[-.6ex]{.13em}{2.3ex}\,}

\newcommand{\lnm}{\left\bracevert}
\newcommand{\rnm}{\right\bracevert}

\newcommand{\p}{\ensuremath{\partial}}
\newcommand{\pr}{\mathcal{P}}

\newcounter{vremennyj}

\newcommand\cond[1]{\setcounter{vremennyj}{\theenumi}\setcounter{enumi}{#1}\labelenumi\setcounter{enumi}{\thevremennyj}}

\newcommand{\pbar}{\ensuremath{\bar{\partial}}}
\newcommand{\db}{\overline\partial}
\newcommand{\D}{\mathbb{D}}
\newcommand{\T}{\mathbb{T}}
\newcommand{\C}{\mathbb{C}}
\newcommand{\N}{\mathbb{N}}
\newcommand{\bP}{\mathbb{P}}

\newcommand{\bS}{\mathbf{S}}
\newcommand{\bk}{\mathbf{k}}

\newcommand\cE{\mathcal{E}}
\newcommand\cP{\mathcal{P}}
\newcommand\cC{\mathcal{C}}
\newcommand\CH{\mathcal{H}}
\newcommand\cU{\mathcal{U}}
\newcommand\cQ{\mathcal{Q}}

\newcommand{\be}{\mathbf{e}}

\newcommand{\la}{w}
\newcommand{\e}{\varepsilon}

\newcommand{\td}{\widetilde\Delta}

\newcommand{\tto}{\!\!\to\!}
\newcommand{\wt}{\widetilde}
\newcommand{\shto}{\raisebox{.3ex}{$\scriptscriptstyle\rightarrow$}\!}

\newcommand{\La}{\langle }
\newcommand{\Ra}{\rangle }
\newcommand{\ran}{\operatorname{ran}}
\newcommand{\tr}{\operatorname{tr}}
\newcommand{\codim}{\operatorname{codim}}
\newcommand\clos{\operatorname{clos}}
\newcommand{\spn}{\operatorname{span}}
\newcommand{\lin}{\operatorname{Lin}}
\newcommand{\rank}{\operatorname{rank}}
\newcommand{\re}{\operatorname{Re}}
\newcommand{\vf}{\varphi}
\newcommand{\f}{\varphi}

%%%%%%%%%%%%%%%%%%%%%%%%%%%%

\newcommand{\entrylabel}[1]{\mbox{#1}\hfill}

\newenvironment{entry}
{\begin{list}{X}%
  {\renewcommand{\makelabel}{\entrylabel}%
      \setlength{\labelwidth}{55pt}%
      \setlength{\leftmargin}{\labelwidth}%\labelsep}%
      \addtolength{\leftmargin}{\labelsep}%
   }%
}%
{\end{list}}

%%%%%%%%%%%%%%%%%%%%%%%%%%%%

%\newenvironment{theorem}{\newline\begin{sloppypar}\noindent\textbf{Theorem}}{\hspace*{\fill}\ %\end{sloppypar}}

\numberwithin{equation}{section}

\newtheorem{thm}{Theorem}[section]
\newtheorem{lm}[thm]{Lemma}
\newtheorem{cor}[thm]{Corollary}
\newtheorem{prop}[thm]{Proposition}
\newtheorem{Definition}[thm]{Definition}
\newtheorem{ex}[thm]{Example}
\newtheorem{conj}[thm]{Conjecture}

\theoremstyle{remark}
\newtheorem{rem}[thm]{Remark}
\newtheorem*{rem*}{Remark}

\title[A subclass of the Cowen-Douglas class and similarity]
{A subclass of the Cowen-Douglas class and similarity}

\author{Kui Ji} \author{Hyun-Kyoung Kwon}  \author{Jaydeb Sarkar} \author{Jing Xu}
\email{jikui@hebtu.edu.cn, hkwon6@albany.edu,jay@isibang.ac.in, xujingmath@outlook.com}
\address{Department of Mathematics, Hebei Normal University, Shijiazhuang, Hebei, 050016, China}
\address{Department of Mathematics and Statistics, University at Albany, State University of New York, Albany, NY, 12222, USA}
\address{Statistics and Mathematics Unit, Indian Statistical Institute, Bangalore, 560059, India}
\address{Department of Mathematics, Hebei Normal University, Shijiazhuang, Hebei, 050016, China}
\thanks{
The first author is supported by the National Natural Science
Foundation of China (Grant No. A010602) and the Foundation for the Author of the National Excellent Doctoral Dissertation of China (Grant No. 201116). The third author is supported in part by the NBHM grant NBHM/R.P.64/2014 and the Mathematical Research Impact Centric Support (MATRICS) grant, File No: MTR/2017/000522 and Core Research Grant, File No: CRG/2019/000908, by the Science and Engineering Research Board (SERB), Department of Science $\&$ Technology (DST) of the Government of India.}



\begin{abstract}
We consider a subclass of the Cowen-Douglas class in which the problem of deciding whether two operators are similar becomes more manageable. A similarity criterion for Cowen-Douglas operators is known to be dependent on the trace of the curvature of the corresponding eigenvector bundles. Unless the given eignvector bundle is a line bundle, the computation of the curvature, in general, is not so simple as one might hope. By using a structure theorem on Cowen-Douglas operators, we reduce the problem of finding the trace of the curvature by looking at the curvatures of the associated line bundles. Several questions related to the similarity problem are also taken into account.
\end{abstract} \maketitle
\setcounter{section}{-1}

\section{Introduction}

Given a complex separable Hilbert space $\mathcal{H}$, let ${\mathcal L}({\mathcal H})$ denote the algebra of bounded linear operators on $\mathcal H$. The set of all $n$-dimensional subspaces of ${\mathcal H}$, called the Grassmannian, will be denoted by
$\mbox{Gr}(n,{\mathcal H})$. When
$\text{dim } {\mathcal H}< \infty$, $\mbox{Gr}(n, {\mathcal H})$ is a
complex manifold. Given a connected open subset $\Omega$ of the complex plane $\mathbb{C}$, M. J. Cowen and R. G. Douglas in \cite{CD},
introduced a class of operators whose point spectra contain the set $\Omega$. More specifically, the class of Cowen-Douglas operators of rank $n$, denoted $B_n(\Omega)$, is defined as follows:
$$\begin{array}{lll}{B}_n(\Omega)=\{T\in \mathcal{L}(\mathcal{H}):
&(1)\,\,\Omega\subset \sigma(T):=\{w\in \mathbb{C}:T-w
\mbox{ is not invertible}\},\\
&(2)\,\,\mbox{dim ker}(T-w)=n \text{ for }w\in\Omega, \\
&(3)\,\,\bigvee_{w\in \Omega}\mbox{ker}(T-w)=\mathcal{H},\text{ and}\\
&(4)\,\,\mbox{ran}(T-w)=\mathcal{H}\text{ for }w\in\Omega\}.
\end{array}$$

It is proven in the same paper that for $T \in B_n(\Omega)$, the mapping from $\Omega$ to $\mbox{Gr}(n, {\mathcal H})$ given by $w \rightarrow \mbox{ker}(T-w)$ defines
$$\mathcal{E}_T=\{(w, x) \in \Omega \times \mathcal{H}: x \in \mbox{ker}(T-w)\},$$ a Hermitian holomorphic vector
bundle of rank $n$ over $\Omega$ with projection $\pi (w, x)=w$. A detailed study of certain aspects of complex geometry is also carried out using the concepts given below.

Following the definition of M. J. Cowen and R. G. Douglas, the curvature function $\mathcal{K}$ for a holomorphic bundle $\mathcal{E}$ of rank $n$ is given by
$$\mathcal{K}(w)
=-\frac{\partial}{\partial \overline{w}}\left(h^{-1}\frac{\partial
h}{\partial w}\right),$$ where
$$h(w)=\left (\langle \gamma_j(w),\gamma_i(w)\rangle \right)_{n\times n},$$ for $w \in \Omega$, denotes the Gram matrix associated with a holomorphic frame $\{\gamma_1,\gamma_2,\cdots,\gamma_n\}$ for $\mathcal{E}$. In the special case of a line bundle (a bundle of rank one), the curvature amounts to calculating
\begin{equation} \label{linebundle}
\mathcal{K}(w)=-\frac{\partial^2}{\partial \overline{w} \partial w} \log\|\gamma(w)\|^2,
\end{equation}
where $\gamma$ denotes a non-vanishing holomorphic cross-section of the bundle $\mathcal{E}$.

Given a $C^{\infty}$ bundle map $\phi$ on a holomorphic vector bundle $\mathcal{E}$ and a holomorphic cross-section $\sigma$ of
$\mathcal{E}$, we have

$\begin{array}{llll}
 &(1)\,\,\phi_{\overline{w}}(\sigma)=\frac{\partial}{\partial \overline{w}}\phi(\sigma), \text{ and}\\[4pt]
  &(2)\,\,\phi_{w}(\sigma)=\frac{\partial}{\partial w}\phi(\sigma)+[h^{-1}\frac{\partial}{\partial w}h,\phi(\sigma)].
\end{array}
$

Since the curvature can be regarded as a bundle map, we obtain the
covariant partial derivatives ${\mathcal
K}_{w^i\overline{w}^j}$ of the curvature $\mathcal
{K}$ by repeatedly using the formulas given above. It is also proven in \cite{CD} that the curvature ${\mathcal K}_T$ and the covariant derivatives ${\mathcal K}_{T, w^i\overline{w}^j}$ of the eigenvector bundle $\mathcal{E}_T$ corresponding to $T \in {B}_n(\Omega)$ form a complete set of unitary invariants.

\begin{thm}[\cite{CD}] Let $T$ and $S$ be Cowen-Douglas operators with Hermitian holormorphic eigenvector bundles $\mathcal{E}_T$ and $\mathcal{E}_S$, respectively. Then $T\sim_{u}
S$ if and only if there exist an isometry $V: \mathcal{E}_{T}\rightarrow
\mathcal{E}_{S}$ and a number $m$ dependent on $\mathcal{E}_T$ and $\mathcal{E}_S$ such that
$$V{\mathcal K}_{T,w^i\overline{w}^j}={\mathcal K}_{S,w^i\overline{w}^j}V,$$
for every $0 \leq i,j \leq m-1$.
\end{thm}

As pointed out by M. J. Cowen and R. G. Douglas, characterizing similarity is a much more intricate issue than describing unitary equivalence. How to make use of the curvature to determine when two Cowen-Douglas operators are similar is still not clear and there have been only some partial results. In \cite{Kwon1}, H. Kwon and
S. Treil gave a similarity theorem to decide when a contraction
operator $T$ is similar to $n$ copies of $M^*_z$, the adjoint of the multiplication operator by $z$, on the
Hardy space of the unit disk $\mathbb{D}$. For a contraction operator $T\in {B}_n(\mathbb{D})$,
let $P(w)$ denote the projection onto the fiber $\ker(T-w)$. Then it is proven that
$T\sim{s}\bigoplus\limits^n M^*_z$ if and only if
$$\bigr\Vert\frac{\partial P(w)}{\partial
w}\bigr\Vert^2_{HS}-\frac{n}{(1-|w|^2)^2}\leq\frac{\partial^{2}}{\partial\overline{w}\partial w}\psi(w),$$ for all
$w\in\mathbb{D}$ and for some bounded subharmonic function $\psi$ defined on $\mathbb{D}$. It is also pointed out that for $n=1$, $\bigr\Vert\frac{\partial P(w)}{\partial w}\bigr\Vert_{HS}^2$, the square of the Hilbert-Schmidt norm of $\frac{\partial P(w)}{\partial
w},$ is the negative of the curvature $\mathcal{K}_T$ of the eigenvector bundle $\mathcal{E}_T$.
Subsequently, the result was generalized from the Hardy shift
to some weighted Bergman shift cases by R. G.
Douglas, H. Kwon, and S. Treil in \cite{Kwon2}. Moreover, in \cite{HJK} and \cite{JS}, $\bigr\Vert \frac{\partial P(w)}{\partial w}\bigr\Vert_{HS}^2$ is proven to be the trace of the curvature $\mathcal{K}_{T}$ when $T\in {B}_n(\Omega)$ and $n$ is an arbitrary positive integer.

For any Cowen-Douglas operator $T$ of rank greater than one, the
curvature ${\mathcal K}_T$ and the corresponding  partial derivatives
${\mathcal K}_{T,w^i\overline{w}^j}$ are not easy
to compute. It is, therefore, necessary to reduce
the number of invariants for Cowen-Douglas operators of higher rank to decide on unitary equivalence or similarity. We first mention the following basic structure theorem proved in the book \cite{JW} that will be relevant for our purpose:

\begin{thm}[\cite{JW}]\label{ut}
For $T \in B_n(\Omega)$, there exist operators $T_{0},T_1,\ldots ,T_{n-1} \in B_1(\Omega)$ and bounded linear operators $S_{i,j}$, $0 \leq i < j \leq n-1$, such that
\renewcommand\arraystretch{0.875}
\begin{equation}\label{1.1T}
T=\left(\begin{matrix}T_{0} & S_{0,1} & S_{0,2}&\cdots&S_{0,n-2}&S_{0,n-1}\\
0 &T_{1}&S_{1,2}&\cdots&S_{1,n-2}&S_{1,n-1} \\
0 &0&T_{2}&\cdots&S_{2,n-2}&S_{2,n-1} \\
\vdots&\vdots&\vdots&\ddots&\vdots&\vdots\\
0&0&0&\cdots&T_{n-2}&S_{n-2,n-1}\\
0&0&0&\cdots&0&T_{n-1}
\end{matrix}\right ).
\end{equation}
\end{thm}
 In \cite{JJKMCR} and \cite{JJKM}, K. Ji, C. Jiang, D. K. Keshari, and G. Misra introduced a subclass $\mathcal{F}B_n(\Omega)$ of the Cowen-Douglas class $B_n(\Omega)$. The class of operators $\mathcal{F}B_n(\Omega)$ is the collection of all $T\in B_n(\Omega)$ with the upper-triangular matrix form given by (\ref{1.1T}), where $T_{i}S_{i,i+1}=S_{i,i+1}T_{i+1}$ and $S_{i, i+1}\neq 0$ for $0\leq i \leq n-2.$ Note that due to this intertwining property, each of the $2\times 2$ block $\Big (\begin{smallmatrix} T_{i} & S_{i,
i+1}\\ 0 & T_{i+1}\end{smallmatrix}\Big)$ in the decomposition of
the operator $T$ is in $\mathcal FB_2(\Omega)$. Hence, by \cite{DM}, the corresponding
second fundamental form $\theta_{i,i+1}(T)$ of $\mathcal{E}_{T_i}$ in $\mathcal{E}_T$ is given by the
formula
\begin{equation} \label{sf}
\theta_{i,i+1}(T)(z) =  \frac{\mathcal K_{T_i}(z) \,d\bar{z}} {\big
(\frac{\|t_{i+1}(z)\|^2} {\|S_{i,i+1}t_{i+1}(z)\|^2}- \mathcal
K_{T_i}(z)\big )^{1/2}},
\end{equation}
where $t_{i+1}$ denotes a non-vanishing section of $\mathcal{E}_{T_{i+1}}.$
For any $T,\widetilde{T}\in \mathcal{F}B_n(\Omega)$ with $\mathcal{K}_{T_i}=\mathcal{K}_{\widetilde{T}_i}$, we have
    $$\theta_{i,i+1}(T)(z) =\theta_{i,i+1}(\widetilde{T})(z) \Leftrightarrow \frac{\|S_{i,i+1}t_{i+1}(z)\|}{\|t_{i+1}(z)\|}=\frac{\|\widetilde{S}_{i,i+1}\widetilde{t}_{i+1}(z)\|}{\|\widetilde{t}_{i+1}(z)\|},$$
so that one can also use $\frac{\|S_{i,i+1}t_{i+1}(z)\|}{\|t_{i+1}(z)\|}$ in place of the second fundamental form $\theta_{i,i+1}(T)$. A unitary classification of
operators in $\mathcal{F}B_n(\Omega)$ is given as follows in terms of the curvature and the second fundamental forms of the corresponding line bundles:

\begin{thm}[\cite{JJKM}]\label{unitary}
For $T, \widetilde{T}\in \mathcal{F}B_n(\Omega)$,
\begin{equation}\nonumber
 T\sim_{u}\widetilde{T}\Leftrightarrow
 \left\{\begin{array}{lll}
\mathcal{K}_{T_i}=\mathcal{K}_{\widetilde{T}_i}\\
 \theta_{i,i+1}(T)=\theta_{i,i+1}(\widetilde{T})\\
 \frac{\langle S_{i,j}(t_{j}), t_i\rangle}{\|t_{i}\|^2}= \frac{\langle \widetilde{S}_{i,j}(\widetilde{t}_{j}), \widetilde{t}_i\rangle}{\|\widetilde{t}_{i}\|^2}\\
 \end{array}
\right\}.
\end{equation}

\end{thm}

In this paper, we obtain a similarity theorem for operators in $\mathcal{F}B_n(\Omega)$ involving the curvatures of the associated line bundles. We first observe that the homogeneity of an operator $T\in\mathcal{F}B_n(\Omega)$ is connected with the similarity problem, the trace of the curvature $\mathcal{K}_T$ can be written as the sum of the curvature $\mathcal{K}_{T_i}$ of the line bundles $\mathcal{E}_{T_i}.$ Note that since it is shown in \cite{JJKM} that operators in $\mathcal{F}B_n(\Omega)$ are irreducible, such a decomposition is non-trivial. Moreover, the $n$-hypercontractivity assumption on the $T_i$, together with an identity that resembles the conditions given in Theorem \ref{unitary} on the second fundamental forms make possible a similarity description in terms of the $\mathcal{K}_{T_i}$. Further results concerning positive definite kernels and the curvature of the tensor product of holomorphic bundles are also presented.

\section{\sf The Base Case $\mathcal{F}B_2(\Omega)$}
We first consider the class $\mathcal{F}B_2(\Omega)$ that will give us information on how to deal with the general case. Let $\mathcal{F}B_2(\Omega)$ denote the set of all bounded linear
operators $T$  of the form $T=\Big ( \begin{smallmatrix}
 T_0 & S \\
 0 & T_1 \\
\end{smallmatrix}\Big ),$
where the two operators $T_0$ and $T_1$ are in the Cowen-Douglas class $B_1(\Omega)$ and the operator $S$ is a non-zero intertwiner between them, that is, $T_0S=ST_1$. It is obvious that if the operators $T_0$ and $T_1$ are defined on separable complex Hilbert spaces $\mathcal H_0$ and $\mathcal{H}_1$, respectively, then $S$ is a non-zero bounded linear operator from $\mathcal H_1$ to $\mathcal H_0$. The operator $T$ is then defined on the Hilbert space $\mathcal H_0\oplus \mathcal H_1.$ Moreover, an operator in $\mathcal FB_2(\Omega)$ obviously belongs to the Cowen-Douglas class $B_2(\Omega)$.

Let ${\mathcal E}_T $ be a holomorphic eigenvector bundle of $T\in\mathcal{F}B_2(\Omega)$ and as usual, let $\text{Hol}(\Omega)$ denote the space of holomorphic functions on $\Omega$. It can then be shown that there exists a holomorphic frame $\{\gamma_0, \gamma_1\}$ of $\mathcal{E}_T$ such that
$$
\gamma_0(w) \perp \left(\frac{\partial}{\partial w}\gamma_0(w)-\gamma_1(w)\right),
$$
for all $w \in \Omega$.  In fact, given any non-zero cross-sections $t_0$ of $\mathcal{E}_{T_0}$ and $t_1$ of $\mathcal{E}_{T_1}$, one sets
$$\gamma_0(w):=\phi(w)t_0(w),$$ for $\phi \in \text{Hol}(\Omega)$ such that $St_1(w)=\phi(w)t_0(w)$ and
$$
\gamma_1(w):=\frac{\partial}{\partial w}\gamma_0(w)-t_1(w)
$$(see \cite{JJKMCR} for details).

Since we will be working with the curvature $\mathcal{K}_T$ of a vector bundle $\mathcal{E}_T$, we mention a related definition.
\begin{Definition} Given a Hermitian
holomorphic vector bundle $\mathcal{E}$ over $\Omega$ of rank $n$ with $\pi:\mathcal{E}\rightarrow \Omega,$ let
$$\wedge^{r}(\mathcal{E}):=\bigcup\limits_{w\in\Omega}\wedge^{r}(\pi^{-1}(w)),
$$ where
$1\leq
r\leq n$ and for $w \in \Omega$, $\wedge^{r}(\pi^{-1}(w))$ denotes the exterior
power space of the fiber $\pi^{-1}(w)$. The space $\wedge^{r}(\pi^{-1}(\mathcal{E}))$ inherits a holomorphic and Hermitian structure from that of $\mathcal{E}$ which makes it a Hermitian holomorphic vector bundle over $\Omega$. When $r=n$, $\wedge^n(\mathcal{E})$ is called the determinant bundle, denoted ${\text{det } \mathcal{E}}$.

\end{Definition}

 Let $\{\gamma_1,\gamma_2,\cdots,\gamma_n\}$ be
 a holomorphic frame for a vector bundle $\mathcal{E}$ on some open set $U\subset \Omega$. Then the wedge product $\gamma_1\wedge
 \gamma_2\wedge\cdots \wedge \gamma_n$ is a frame for ${\text{det}\mathcal{E}}$ over $U$. If we denote by $h_{\text{det}\mathcal{E}}$ the corresponding Gram matrix, then
$$h_{{\text{det }\mathcal{E}}}={\text{det}}h_{\mathcal{E}}.$$
In particular, given a holomorphic frame $\sigma=\{\gamma\}$ of $\mathcal{E}$ on $\Omega$, a holomorphic frame for
the 1-jet bundle ${\mathcal J}_1(\mathcal{E})$ is given by
$${\mathcal J}_1(\sigma)=\{\gamma,\frac{\partial}{\partial w}\gamma\},$$ and the Gram matrix
$h(w)= \langle \gamma(w),\gamma(w) \rangle$ for $w \in \Omega$ induces the following Gram matrix ${\mathcal
J}_1(h)$ for ${\mathcal J}_1(\mathcal{E})$:
$$\begin{array}{lll}
{\mathcal
J}_1(h)(w)&=&\left(\begin{array}{ccccc}
\langle \gamma(w),\gamma(w) \rangle &
\frac{\partial}{\partial w} \langle \gamma(w),\gamma(w) \rangle \\
\frac{\partial}{\partial \overline{w}}\langle \gamma(w),\gamma(w) \rangle &\frac{\partial^{2}}{\partial\overline{w}\partial w }\langle \gamma(w),\gamma(w) \rangle \end{array}\right)\\
&=&\left(\begin{array}{ccccc}h(w)&
\frac{\partial}{\partial w} h(w)\\ \frac{\partial}{\partial \overline{w}}
h(w)&\frac{\partial^{2}}{\partial\overline{w}\partial w } h(w)
\end{array}\right).\end{array}$$

The relationship between the curvature of the determinant bundle $\mathcal{E}$ and that of the vector bundle $\mathcal{E}$ is well-known (see \cite{CD} and \cite{Demailly}). Recently, D. K. Keshari give an elementary and detailed proof of this relationship in \cite{Dinesh}.
\begin{lm}[\cite{CD},\cite{Demailly},\cite{Dinesh}]\label{Dinesh} Let $\mathcal{E}$ be a Hermitian holomorphic vector bundle over $\Omega$ of rank $n$ with $\pi:\mathcal{E} \rightarrow \Omega$. Then for $w \in \Omega$,
$${\mathcal K}_{\text{det }\mathcal{E}}(w) = \text{trace }{\mathcal
K}_{\mathcal{E}}(w).$$
\end{lm}

We now investigate situations in which the trace of the curvature $\mathcal{K}_T$ for $T=\Big ( \begin{smallmatrix}
 T_0 & S_{0,1} \\
 0 & T_1 \\
\end{smallmatrix}\Big ) \in \mathcal{F}B_2(\Omega)$ can be computed using the curvatures of the operators $T_0$ and $T_1$. Recall that the curvature of the line bundles $\mathcal{E}_{T_0}$ and $\mathcal{E}_{T_1}$ are easily found using expression (\ref{linebundle}). We start with a simple lemma.

\begin{lm}\label{trace1}For $T=\Big ( \begin{smallmatrix}
 T_0 & S_{0,1} \\
 0 & T_1 \\
\end{smallmatrix}\Big ) \in \mathcal{F}B_2(\Omega)$, let $\{\gamma_0, \gamma_1\}$ be a holomorphic frame of $\mathcal{E}_T$ such that
$$
\gamma_0(w) \perp \left (\frac{\partial}{\partial w}\gamma_0(w)-\gamma_1(w)\right).
$$
Then for every $w \in \Omega$, $$\text{trace }{\mathcal K}_{T}(w)={\mathcal
K}_{T_0}(w)-\frac{\partial^{2}}{\partial\overline{w}\partial w }\log \Big (h_1(w)-{\mathcal K}_{T_0}(w)h_0(w) \Big),$$ where
$h_0(w)=||\gamma_0(w)||^2$ and $h_1(w)=\big\|\frac{\partial}{\partial w}\gamma_0(w)-\gamma_1(w)\big\|^2$.
\end{lm}

\begin{proof}
Let $h_{\mathcal{E}}$ be the Gram matrix of the frame $\{\gamma_0, \gamma_1\},$ we have
$$ \begin{array}{llll}
h_{\mathcal{E}}(w)
&=&\left(\begin{array}{ccc}
h_0(w)&\frac{\partial}{\partial w} h_0(w)\\
\frac{\partial}{{\partial}{\overline w}} h_0(w)&
\frac{\partial^2}{\partial {\overline w} \partial w} h_0(w)
\end{array}\right)+\left(\begin{array}{ccc}
0&0\\
0&h_1(w)
\end{array}\right),\\
\end{array}$$
where
$h_0(w)=||\gamma_0(w)||^2$ and $h_1(w)=\big\|\frac{\partial}{\partial w}\gamma_0(w)-\gamma_1(w)\big\|^2$.
Then we know from Lemma \ref{Dinesh} that
$$
\text{trace}{\mathcal K}_{T}(w) ={\mathcal K}_{\text{detT}}(w)={\mathcal K}_{T_0}(w)-\frac{\partial^{2}}{\partial\overline{w}\partial w }\log\bigg(h_1(w)-{\mathcal K}_{T_0}(w)h_0(w)\bigg).
$$
\end{proof}

The following proposition is a direct consequence of Lemma \ref{trace1}:

\begin{prop} \label{trace 2}
Let $T=\Big ( \begin{smallmatrix}
 T_0 & S_{0,1} \\
 0 & T_1 \\
\end{smallmatrix}\Big ) \in \mathcal{F}B_2(\Omega)$. Then $\text{trace}{\mathcal K}_{T}={\mathcal
K}_{T_0}+{\mathcal K}_{T_1}$ if and only if there exists some $\phi \in \text{Hol}(\Omega)$ with $|\phi(w)|>1$ for all $w \in \Omega$ such that $${\mathcal K}_{T_0}=\frac{|\phi|^2}{1-|\phi|^2}\theta^2_{0,1}(T).$$

\end{prop}

\begin{proof} Consider the frame $\{-S_{0,1}t, -\frac{\partial}{\partial w} S_{0,1}t+t\}$ for $\mathcal{E}_T$, where $t$ is a cross-section of $\mathcal{E}_{T_1}$. Let $h_0(w)=\|-S_{0,1}t(w)\|^2$ and $h_1(w)=\|-\frac{\partial}{\partial w} S_{0,1}t+t\|^2$. Then by Lemma \ref{trace1},  we have
$$\text{trace}{\mathcal K}_{T}={\mathcal K}_{T_0}-\frac{\partial^{2}}{\partial\overline{w}\partial w }\log(h_1-{\mathcal K}_{T_0}h_0).$$
If $\text{trace}{\mathcal K}_{T}={\mathcal
K}_{T_0}+{\mathcal K}_{T_1}$, then obviously,
$$\frac{\partial^{2}}{\partial\overline{w}\partial w } \log \left(\frac{h_1-{\mathcal K}_{T_0}h_0}{h_1}\right )^{\frac{1}{2}}=0.$$ Since the function $$u:= \log \left (\frac{h_1-{\mathcal K}_{T_0}h_0}{h_1}\right)^{\frac{1}{2}}$$ is real-valued and harmonic, setting $$\phi:=e^{u+iv} \in \text{Hol}(\Omega),$$ where $v$ is the conjugate harmonic of $u$, it follows that $$|\phi|=e^u=\left(\frac{h_1-{\mathcal K}_{T_0}h_0}{h_1}\right)^{\frac{1}{2}}.$$
 Notice that since ${\mathcal K}_{T_0}(w)<0$ for all $w\in \Omega$, $|\phi(w)|>1$ and  ${\mathcal K}_{T_0}=(1-|\phi|^2)\frac{h_1}{h_0}.$ Then by formula (\ref{sf}),
$$\begin{array}{llll}\theta_{0,1}(T)&=&  \frac{\mathcal K_{T_0}} {\big
(\frac{\|t\|^2}{\|S_{0,1}t\|^2} - \mathcal
K_{T_0}\big )^{1/2}}=\frac{\mathcal K_{T_0}} {\big
(\frac{h_1}{h_0} - \mathcal
K_{T_0}\big )^{1/2}}=\frac{\mathcal K_{T_0}} {\big
(\frac{1}{1-|\phi|^2}{\mathcal K}_{T_0} - \mathcal
K_{T_0}\big )^{1/2}},
\end{array}$$ so that ${\mathcal K}_{T_0}=\frac{|\phi|^2}{1-|\phi|^2}\theta^2_{0,1}(T).$

On the other hand, suppose that  ${\mathcal K}_{T_0}=\frac{|\phi|^2}{1-|\phi|^2}\theta^2_{0,1}(T).$ Then since
$${\mathcal K}_{T_0}=(1-|\phi|^2)\frac{h_1}{h_0},$$ we have
$$\begin{array}{lll}
\text{trace}{\mathcal K}_{T}&=&{\mathcal K}_{T_0}-\frac{\partial^{2}}{\partial\overline{w}\partial w }\log (h_1-{\mathcal K}_{T_0}h_0)\\
&=&{\mathcal K}_{T_0}-\frac{\partial^{2}}{\partial\overline{w}\partial w }\log(|\phi|^2h_1)\\
&=&{\mathcal K}_{T_0}+{\mathcal K}_{T_1}.
\end{array}$$

\end{proof}

The following result characterizes homogeneous operators in $\mathcal{F}B_2(\mathbb{D})$. Recall that a bounded operator $T$ is said to be \emph{homogeneous} if  for all linear fractional transformations $\varphi$ from $\mathbb{D}$ onto $\mathbb{D}$ that are analytic on $\sigma(T)$, $\varphi(T)$ is unitarily equivalent to $T$.

\begin{lm}[\cite{JJKM}]\label{JJKM1}
An operator $T=\Big ( \begin{smallmatrix}
 T_0 & S_{0,1} \\
 0 & T_1 \\
\end{smallmatrix}\Big ) \in \mathcal{F}B_2(\mathbb{D})$ is homogeneous if and only if
\begin{enumerate}
\item[(1)] $T_0$ and $T_1$ are homogeneous operators, \item[(2)]
$\mathcal{K}_{T_1}(w)=\mathcal{K}_{T_0}(w)+\mathcal{K}_{B^*}(w)$
for every $w\in\mathbb{D}$, where $B$ denotes the Bergman shift operator, \text{ and}

\item[(3)] There exist non-vanishing holomorphic cross-sections $t_0$ and $t_1$ for $\mathcal{E}_{T_0}$ and $\mathcal{E}_{T_1}$, respectively, a constant $a > 0$, and an $\alpha \in \mathbb{N}$ such that
 $\|t_0(w)\|^2=\frac{1}{(1-|w|^2)^{\alpha}}$,
 $\|t_1(w)\|^2=\frac{1}{(1-|w|^2)^{\alpha+2}},$
 and $S_{0,1}t_1(w)=at_0(w).$
\end{enumerate}
\end{lm}

Given a homogeneous operator
$T\in\mathcal{F}{B}_{2}(\mathbb{D})$, we can assume by Lemma \ref{JJKM1} that
$$t_0(w)=\frac{1}{(1-zw)^{\alpha}} \text{ and } t_1(w)=\frac{1}{(1-zw)^{\alpha+2}},$$
for some $\alpha \in \mathbb{N}$, and that $T_0$ is the backward shift operator $M^*_z$ on the Hilbert space of analytic functions $f$ on $\mathbb{D}$ such that
$$
\sum_{k=0}^{\infty} |\hat{f}(k)|^2 \frac{1}{{{\alpha+k-1} \choose k}}< \infty.
$$
The operator $T_1$ can also be viewed as $M^*_z$ on a related Hilbert space.
Since a holomorphic frame of $\mathcal{E}_T$ is also given by
$$\begin{array}{llll}
\gamma_0&=&t_0\\
\gamma_{1}&=&\frac{\partial}{\partial w}t_0-\frac{1}{a}t_{1},\\
\end{array}$$
one can even consider a more general operator $T\in
\mathcal{F}{B}_{2}(\mathbb{D})$ whose eigenvector bundle $\mathcal{E}_T$ possesses a holomorphic frame of the form
$$\begin{array}{llll}\gamma_0&=&t_0\\
                     \gamma_{1}&=&\frac{\partial}{\partial w}t_0+\phi t_{1},\\
\end{array}$$
for some
$t_0(w)=\frac{1}{(1-zw)^{\alpha_0}}$ and $t_1(w)=\frac{1}{(1-zw)^{\alpha_1}}$, where
$\alpha_0+2\geq\alpha_1>\alpha_0$, and for some $\phi \in
GL(H^{\infty}(\mathbb{D})).$ $GL(H^{\infty}(\mathbb{D}))$ as usual, stands for the general linear group over the space of bounded analytic functions on $\mathbb{D}$.  These kinds of operators are said to be \emph{quasi-homogeneous}.

We next show that for a homogeneous operator $T$ in $\mathcal{F}B_2(\mathbb{D})$, it becomes a simple matter to find $\text{trace}\mathcal{K}_T$.
\begin{prop} \label{hm1}Let $T=\Big ( \begin{smallmatrix}
 T_0 & S_{0,1} \\
 0 & T_1 \\
\end{smallmatrix}\Big ) \in \mathcal{F}B_2(\mathbb{D})$ be a homogeneous operator. Then
$$\text{trace}{\mathcal K}_{T}={\mathcal K}_{T_0}+{\mathcal K}_{T_1}.$$
\end{prop}

\begin{proof}
Since $T$ is homogeneous, there exist constants $a >0$ and $\alpha \in \mathbb{N}$ such that
$$\begin{array}{llll}\gamma_0&=&a\frac{1}{(1-zw)^{\alpha}}\\
                     \gamma_{1}&=&a\frac{\partial}{\partial w}\left(\frac{1}{(1-zw)^{\alpha}} \right)-\frac{1}{(1-zw)^{\alpha+2}},\\
\end{array}$$
form a frame for ${\mathcal
E}_{T}$. Then
 $$h(w)=\left(\begin{array}{ccccc}
 h_0(w)&\frac{\partial}{\partial w} h_0(w)\\
\frac{\partial} {{\partial} \overline{w}} h_0(w)&
\frac{\partial^2}{{\partial}\overline{w} \partial {w}} h_0(w)+h_1(w)\\
\end{array}\right).$$
where $h_i(w)=\|\gamma_i(w)\|^2 (i=1,2)$. Since $\text{trace}{\mathcal K}_{T}(w)={\mathcal K}_{\text{detT}}(w)=-\frac{2\alpha+2}{(1-|w|^2)^2},$ the proof is complete.
\end{proof}

By using the methods similar to the ones used in \cite{Dinesh}, we can generalize Proposition \ref{hm1}  to homogeneous operators that belong to $\mathcal{F}{B}_{3}(\mathbb{D})$. The proof is omitted since we have not been able to generalize the computations involved in this particular case. We infer that the result holds for every $n\in\mathbb{N}$.

\begin{prop}
For $T\in \mathcal{F}{B}_{3}(\mathbb{D})$ that is a homogeneous operator, we have for all $w\in\mathbb{D}$,
$$\text{trace}{\mathcal
K}_{T}(w)={\mathcal K}_{T_0}(w)+{\mathcal K}_{T_1}(w)+{\mathcal
K}_{T_2}(w).$$
\end{prop}

\begin{conj}
Let $T\in \mathcal{F}{B}_{n}(\mathbb{D})$ be a homogeneous operator, then for all $w\in\mathbb{D}$,
$$\text{trace}{\mathcal
K}_{T}(w)={\mathcal K}_{T_0}(w)+{\mathcal K}_{T_1}(w)+\cdots+{\mathcal
K}_{T_{n-1}}(w).$$
\end{conj}

\begin{rem}\label{constant}
By combining Propositions \ref{trace 2} and \ref{hm1}, we see that for a homogeneous operator
$T=\Big ( \begin{smallmatrix}
 T_0 & S_{0,1} \\
 0 & T_1 \\
\end{smallmatrix}\Big ) \in \mathcal{F}B_2(\Omega)$, there exists a $\phi \in \text{Hol}(\Omega)$ with $${\mathcal K}_{T_0}=\frac{|\phi|^2}{1-|\phi|^2}\theta^2_{0,1}(T).$$ In fact, one can take $\phi$ to be the constant function
$$
\phi(w)=\left( 1+\alpha |a|^2 \right)^{\frac{1}{2}}.
$$

\end{rem}

We now show that the condition $$\text{trace}{\mathcal K}_{T}={\mathcal K}_{T_0}+{\mathcal K}_{T_1}$$ can also be used to say something about the similarity of operators in $\mathcal{F}B_2(\mathbb{D})$. The following lemma is well-known, and can be found in \cite{FV}, for instance.

\begin{lm} \label{hm2}
Let $f \in \text{Hol}(\Omega)$ be a function on $\Omega$ taking values in a Hilbert space. If $\|f(w)\|^{2}=1$ for all $w \in \Omega$, then $f$ is a constant function.
\end{lm}

\begin{prop} Let $T=\Big ( \begin{smallmatrix}
 T_0 & S_{0,1} \\
 0 & T_1 \\
\end{smallmatrix}\Big ) \in \mathcal{F}B_2(\Omega)$ be a homogeneous operator. If $ \widetilde{T}=\Big ( \begin{smallmatrix}
 T_0 & \widetilde{S}_{0,1} \\
 0 & T_1 \\
\end{smallmatrix}\Big ) \in \mathcal{F}B_2(\Omega)$
 is such that $\text{trace}{\mathcal K}_{\widetilde{T}}={\mathcal K}_{T_0}+{\mathcal K}_{T_1}$,  then
 $T\sim_s \widetilde{T}.$

\end{prop}

\begin{proof} Let
$\{t_0,\frac{\partial}{\partial w}t_0+t_1\}$ be a holomorphic frame for $\mathcal{E}_T$ with $S_{0,1}t_1=-t_0$. Notice that $$\widetilde{S}_{0,1}t_1=-\psi t_0,$$ for some $\psi \in \text{Hol}(\Omega)$ and that $\text{trace}{\mathcal K}_{T}=\text{trace}{\mathcal K}_{\widetilde{T}}={\mathcal K}_{T_0}+{\mathcal K}_{T_1}.$ Then by Remark \ref{constant}, there exist constant functions $\phi$ and $\widetilde{\phi}$ on $\Omega$ with $|\phi(w)|^2, |\widetilde{\phi}(w)|^2>1$ such that
$${\mathcal K}_{T_0}=\frac{|\phi|^2}{1-|\phi|^2}\theta^2_{0,1}(T)=\frac{|\widetilde{\phi}|^2}{1-|\widetilde{\phi}|^2}\theta^2_{0,1}(\widetilde{T}).$$
This implies that $(1-|\phi|^2)\frac{h_1}{h_0}=(1-|\widetilde{\phi}|^2)\frac{h_1}{|\psi|^2h_0},$ where as before, $h_i(w)=\|t_i(w)\|^2$. If we set $c=1-|\phi|^2$, then $$c|\psi(w)|^2+|\widetilde{\phi}(w)|^2=1,$$ for all $w\in \mathbb{D}$.
Applying $\frac{\partial}{\partial \overline{w}}$ to both sides, we have $c\psi(w)\frac{\partial}{\partial \overline{w}}\overline{\psi}(w)+\widetilde{\phi}(w)\frac{\partial}{\partial \overline{w}}\overline{\tilde{\phi}}(w)=0.$ Then the  meromorphic function $\frac{c\psi}{\widetilde{\phi}}$ is equal to the anti-meromorphic function $-\frac{\frac{\partial}{\partial \overline{w}}\overline{\widetilde{\phi}}}{\frac{\partial}{\partial \overline{w}}\overline{\psi}},$ so that $\frac{c\psi}{\widetilde{\phi}}$ is a constant. It follows that $\psi$ is also a constant, and by Lemma \ref{JJKM1}, we conclude that $\widetilde{T}$ is homogeneous.

Now define a bundle map $\Phi: \mathcal{E}_{T_1} \rightarrow \mathcal{E}_{T_1}$ as
$$\Phi(t_1(w))=\psi t_1(w),$$ for each $w\in \mathbb{D}.$
Since $\psi\neq 0$ is a constant, the map $\Phi$ induces an invertible operator in the commutant $\{T_1\}^{\prime}$ of $T_1$ and we denote this operator by $X_1$.  Then since
$$S_{0,1}X_1t_1(w)=S_{0,1}(\psi t_1(w))=-\psi t_0(w)=\widetilde{S}_{0,1}t_1(w),$$ for all $w \in \Omega$, $$\widetilde{S}_{0,1}=S_{0,1}X_1.$$  Now setting $X=\Big (\begin{smallmatrix}I & 0 \\
0 & X_1 \\
\end{smallmatrix}\Big ),$
we conclude that $X$ is invertible and that
$$\Big (\begin{smallmatrix}I & 0 \\
0 & X_1 \\
\end{smallmatrix}\Big )\Big ( \begin{smallmatrix}
 T_0 & \widetilde{S}_{0,1} \\
 0 & T_1 \\
\end{smallmatrix}\Big ) =\Big ( \begin{smallmatrix}
 T_0 & S_{0,1} \\
 0 & T_1 \\
\end{smallmatrix}\Big ) \Big (\begin{smallmatrix}I & 0 \\
0 & X_1 \\
\end{smallmatrix}\Big ).$$
\end{proof}

\begin{rem}
The homogeneity of an operator is preserved under a unitary transformation and thus, $\widetilde{T}=\Big(\begin{smallmatrix}
 T_0 & \widetilde{S}_{0,1} \\
 0 & T_1 \\
\end{smallmatrix}\Big)\in\mathcal{F}B_2(\Omega)$ is unitarily equivalent to a homogeneous operator if and only if $\widetilde{T}$ itself is homogeneous.
\end{rem}

We now give several equivalent statements to the condition $\text{trace}\mathcal{K}_{T}=\mathcal{K}_{T_0}+\mathcal{K}_{T_1}.$

\begin{thm}\label{hm3} Let $T=\Big ( \begin{smallmatrix}
 T_0 & S_{0,1} \\
 0 & T_1 \\
\end{smallmatrix}\Big ) \in \mathcal{F}B_2(\mathbb{D})$ and suppose that $f \in \text{Hol}(\mathbb{D})$ takes values in a Hilbert space $\mathcal{H}$. Let $\gamma_0$ and $\gamma_1$ be the non-vanishing holomorphic cross-sections of ${\mathcal E}_{T_{0}}$ and $\mathcal{E}_{T_1}$, respectively, such that $\gamma_0(w) \perp (\frac{\partial}{\partial w}\gamma_0(w)-\gamma_1(w))$. Set $h_{i}(w)=\|\gamma_{i}(w)\|^{2}$ as before and suppose that for all $w \in \mathbb{D}$, one of the following conditions hold:
\begin{enumerate}
\item $-\left(\mathcal{K}_{T_{0}}\frac{h_{0}}{h_{1}}\right)(w)=\|f(w)\|^{2}$, or
\item $-\left ({\mathcal K}_{T_{0}}\frac{h_{0}}{h_{1}}\right)(w)=\|f(w)\|^{-2}$ and $\lim\limits_{|w|\rightarrow1^{-}}\|f(w)\|^{2}=\infty.$
\end{enumerate}
Then  $\text{trace}{\mathcal K}_{T}={\mathcal K}_{T_0}+{\mathcal K}_{T_1}$ if and only if for some $\lambda > 0$, $h_{1}=\lambda(-{\mathcal K}_{T_0}h_0)$.

\end{thm}

\begin{proof}
If $h$ denotes the Gram matrix

$$\begin{array}{lllll}
h(w)&=&\left(\begin{array}{ccccc}h_0(w)&
 \frac{\partial}{\partial w} h_0(w)\\
\frac{\partial}{{\partial} \overline{w}} h_0(w)&
\frac{\partial^2}{{\partial}\overline{w}\partial w} h_0(w)
+h_1(w)
\end{array}\right),\\
\end{array}
$$
by Lemma \ref{trace1}, we have
$$\text{trace}{\mathcal K}_{T}(w)={\mathcal K}_{T_0}(w)- \frac{\partial^2}{{\partial}\overline{w}\partial w}\log \bigg(h_1(w)-{\mathcal K}_{T_0}(w)h_0(w) \bigg).$$
If $\text{trace}{\mathcal K}_{T}={\mathcal K}_{T_0}+{\mathcal K}_{T_1},$ then $\frac{\partial^2}{{\partial}\overline{w}\partial w}\log \left(\frac{h_{1}-{\mathcal K}_{T_0}h_{0}}{h_{1}} \right)=0,$ and therefore, there exists $\phi\in\text{Hol}(\mathbb{D})$ such that $\frac{h_{1}-{\mathcal K}_{T_0}h_{0}}{h_{1}}=|\phi|^{2}.$

We first consider the condition $-\left(\mathcal{K}_{T_{0}}\frac{h_{0}}{h_{1}}\right)(w)=\Vert f(w)\Vert^{2},$ which implies $$1+\Vert f(w)\Vert^{2}=|\phi(w)|^{2},$$ and hence, $\|f'(w)\|^2=\phi^{'}(w)\overline{\phi^{'}}(w).$
 If $\phi^{'}=0$, then $\phi$ is a constant function. If not, we assume that  $\phi^{'}(w)\neq0$ by considering the open set $\{w \in \mathbb{D}: \phi(w) \neq 0\}$ instead of $\mathbb{D}$. We then have $\big\|\frac{f'(w)}{\phi^{'}(w)}\big\|=1$. It follows using Lemma \ref{hm2} that $\frac{f'(w)}{\phi^{'}(w)}=c,$ for a constant $c$ of length $1$. Then $f(w)=c\phi(w)+d$ for some $d\in \mathcal{H}$ and therefore,
 $$\begin{array}{llll}0&=&1+\Vert c\phi(w)+d\Vert^{2}-|\phi(w)|^{2}\\
&=&1+|c|^{2}|\phi(w)|^{2}+\phi(w)\langle c,d \rangle+\overline{\phi}(w)\langle d,c \rangle+\Vert d\Vert^{2}-|\phi(w)|^{2}\\
&=&1+\phi(w) \langle c,d \rangle +\overline{\phi}(w) \langle d,c \rangle+\Vert d\Vert^{2}. \\
\end{array}$$
Applying $\frac{\partial}{\partial w}$ to the above, we have $ \langle c,d \rangle=0,$ and hence $\Vert d\Vert^{2}+1=0,$ which is a contradiction. Thus $\phi(w)$ is a constant function, also making $\Vert f(w)\Vert^{2}=|\phi(w)|^{2}-1$ constant. Letting $\lambda=\frac{1}{\Vert f(w)\Vert^{2}}>0$, we have $h_{1}=\lambda(-{\mathcal K}_{T_0}h_0).$

We now consider the second condition of the theorem. If $\text{trace}{\mathcal K}_{T}={\mathcal K}_{T_0}+{\mathcal K}_{T_1}$ and $-\left(\mathcal{K}_{T_{0}}\frac{h_{0}}{h_{1}}\right)(w)=\Vert f(w)\Vert^{-2},$ we get $\Vert f(w)\Vert^{-2}=|\phi(w)|^{2}-1>0$
and $$\Vert f(w)\Vert^{2}=\frac{1}{|\phi(w)|^{2}-1}=\frac{1}{|\phi(w)|^{2}}\left(\frac{1}{1-|\phi(w)|^{-2}}\right)
=\frac{1}{|\phi(w)|^{2}}\sum\limits_{n=0}^{\infty}\frac{1}{|\phi(w)|^{2n}}.$$
Let $f(w)=\frac{1}{\phi(w)}\left(\sum\limits_{n=0}^{\infty}\frac{1}{\phi^n(w)}e_{n}\right)$, where $\{e_n\}_{n=0}^{\infty}$ is an orthonormal basis of $\mathcal{H}$. Then since $\lim\limits_{|w|\rightarrow1^{-}}\Vert f(w)\Vert^{2}=\infty,$ $$\lim\limits_{|w|\rightarrow1^{-}}|\phi(w)|^{2}=\lim\limits_{|w|\rightarrow1^{-}}\Vert f(w)\Vert^{-2}+1=1,$$
and it follows that since $|\phi(w)|>1$ for all $w \in \mathbb{D}$, the function $\phi$ is constant. If we let $\lambda^{-1}=|\phi|^{2}-1>0,$ then $h_{1}=\lambda(-{\mathcal K}_{T_0}h_0).$

Conversely, if $h_{1}=\lambda(-{\mathcal K}_{T_0}h_0)$ for some $\lambda>0$, then $\frac{\partial^2}{{\partial}\overline{w}\partial w}\log(\frac{h_{1}-{\mathcal K}_{T_0}h_{0}}{h_{1}})=0.$ Since $\text{trace}{\mathcal K}_{T}={\mathcal K}_{T_0}- \frac{\partial^2}{{\partial}\overline{w}\partial w}\log(h_1-{\mathcal K}_{T_0}h_0),$ we know that $\text{trace}{\mathcal K}_{T}={\mathcal K}_{T_0}+{\mathcal K}_{T_1}.$
\end{proof}

\begin{cor}
Let $T=\Big ( \begin{smallmatrix}
 T_0 & S_{0,1} \\
 0 & T_1 \\
\end{smallmatrix}\Big ) \in  \mathcal{F}B_2(\mathbb{D}).$ Suppose that $T_{i}\sim_{u}(M_z^*, \mathcal{H}_{K_i}),$ where the Hilbert space $\mathcal{H}_{K_i}$ has a reproducing kernel of the form $K_{i}(z,\omega)=\frac{1}{(1-z\overline{\omega})^{\lambda_{i}}}$ for some $\lambda_{i} \in \mathbb{N}$. Then $\text{trace}{\mathcal K}_{T}={\mathcal K}_{T_0}+{\mathcal K}_{T_1}$ if and only if $\lambda_{1}=\lambda_{0}+2.$
\end{cor}

\begin{proof}
Since ${K}_{i}(z,\omega)=\frac{1}{(1-z\overline{\omega})^{\lambda_{i}}}$, $h_i(w)=\frac{1}{(1-|\omega|^{2})^{\lambda_{i}}},$ and ${\mathcal K}_{T_0}(w)=-\frac{\lambda_{0}}{(1-|\omega|^{2})^{2}}.$
Then $$-\left({\mathcal K}_{T_{0}}\frac{h_{0}}{h_{1}}\right)(w)=\lambda_{0}(1-|\omega|^{2})^{\lambda_{1}-(\lambda_{0}+2)},$$
and therefore by Theorem \ref{hm3}, $\text{trace}{\mathcal K}_{T}={\mathcal K}_{T_0}+{\mathcal K}_{T_1}$ if and only if $-{\mathcal K}_{T_{0}}\frac{h_{0}}{h_{1}}$ is a constant, that is, $\lambda_{1}=\lambda_{0}+2.$
\end{proof}

\section{On the equation $\frac{\partial^2}{{\partial z} \partial \overline w}\log K(z,w)=[K(z,w)]^{p}$}

In Theorem \ref{hm3}, we encountered the condition $\|\gamma_1(w)\|^2=\lambda \|\gamma_0(w)\|^2\frac{\partial^{2}}{\partial\overline{w}\partial w}\log \|\gamma_0(w)\|^2.$ An associated question that has been raised by G. Misra is as follows:

Let $K :\mathbb{D}\times \mathbb{D} \rightarrow \mathbb{C}$ be a sesqui-analytic function. When is the function
$K(z, w)\frac{\partial^2}{{\partial z} \partial \overline w}\log{K}(z,w)$ a positive definite kernel?

One can come up with several counterexamples to show that $K(z,w)\frac{\partial^2}{{\partial z} \partial \overline w}\log K(z,w)$ need not be a positive definite kernel. A simple case giving an affirmative answer occurs when one sets $K=K^{\alpha}K^{\beta},$ where both $K^{\alpha}$ and $K^{\beta}$ are positive definite kernels. We give a necessary and sufficient condition for the equation $\frac{\partial^2}{{\partial z} \partial \overline w}\log K(z,w)=[K(z,w)]^{p}$ for some $p \in \mathbb{N}$ to hold for a diagonal reproducing kernel. At this point, we note that $K(z, w)\frac{\partial^2}{{\partial z} \partial \overline w}\log{K}(z,w)$ is a positive definite kernel, and give a special sufficient condition for the open question raised by G. Misra. We first start with a necessary condition for $K(z,w) \frac{\partial^2}{{\partial z} \partial \overline w}\log K(z,w)$ to be a positive definite kernel.

\begin{prop}\label{tf5}
Given a positive definite kernel $K(z,w)=1+\sum\limits_{i=1}^{\infty}a_{i}z^{i}\overline w^{i}$ on $\mathbb{D}\times \mathbb{D}$, if \\
$K(z,w)\frac{\partial^2}{{\partial z} \partial \overline w}\log K(z,w)$ is a positive definite kernel, then for any $n \in \mathbb{N},$ $$a_{n+1}\geq-\frac{1}{(n+1)^{2}}\left (\sum\limits_{i=1}^{n}i^{2}a_{n+1-i}a_{i}+
\sum\limits_{i=2}^{n+1}\sum\limits_{k=2}^{i}(-1)^{k-1}\frac{i^{2}}{k}
\left [\sum\limits_{\sum\limits_{j=1}^{k}l_{j}=i}a_{n+1-i}(\prod\limits_{j=1}^{k}a_{l_{j}})\right]\right).$$
\end{prop}

\begin{proof}
Setting
$$b_{n}:=\sum\limits_{k=1}^{n}(-1)^{k-1}\frac{1}{k}
\left (\sum\limits_{\sum\limits_{j=1}^{k}i_{j}=n}(\prod\limits_{j=1}^{k}a_{i_{j}})\right),\quad n\geq1,$$
we have
$\frac{\partial^2}{\partial \overline{w} \partial w}\log K (w,w)=\sum\limits_{n=1}^{\infty}n^{2}b_{n}|w|^{2(n-1)}.$
Then
$$\begin{array}{lll}
K(w, w)\frac{\partial^2}{\partial \overline{w} \partial w}\log K(w,w)&=&\left(1+\sum\limits_{i=1}^{\infty}a_{i}|w|^{2i}\right)
\left(\sum\limits_{n=1}^{\infty}n^{2}b_{n}|w|^{2(n-1)}\right)\\[4pt]
&=&b_{1}+\sum\limits_{k=1}^{\infty}\left((k+1)^{2}b_{k+1}+\sum\limits_{i=1}^{k}i^{2}a_{k+1-i}b_{i}\right)|w|^{2k}.
\end{array}$$
Note that for $n \geq 1$, the coefficient of $|w|^{2n}$ is given by
$$\begin{array}{llll}
&(n+1)^{2}b_{n+1}+\sum\limits_{i=1}^{n}i^{2}a_{n+1-i}b_{i}\\
=&(n+1)^{2}\left [\sum\limits_{k=1}^{n+1}(-1)^{k-1}\frac{1}{k}
\left(\sum\limits_{\sum\limits_{j=1}^{k}i_{j}=n+1}(\prod\limits_{j=1}^{k}a_{i_{j}})\right)\right]+
\sum\limits_{i=1}^{n}i^{2}a_{n+1-i}\left[\sum\limits_{k=1}^{i}(-1)^{k-1}\frac{1}{k}
\left(\sum\limits_{\sum\limits_{j=1}^{k}l_{j}=i}(\prod\limits_{j=1}^{k}a_{l_{j}})\right)\right]\\[4pt]
=&(n+1)^{2}a_{n+1}+(n+1)^{2}\left[\sum\limits_{k=2}^{n+1}(-1)^{k-1}\frac{1}{k}
\left(\sum\limits_{\sum\limits_{j=1}^{k}i_{j}=n+1}(\prod\limits_{j=1}^{k}a_{i_{j}})\right)\right]\\
&+\sum\limits_{i=1}^{n}i^{2}a_{n+1-i}\left[\sum\limits_{k=1}^{i}(-1)^{k-1}\frac{1}{k}
\left(\sum\limits_{\sum\limits_{j=1}^{k}l_{j}=i}(\prod\limits_{j=1}^{k}a_{l_{j}})\right)\right].
\end{array}$$
Assuming $a_0=1$, without loss of generality, we have
$$\begin{array}{lllll}
a_{n+1}&\geq&-\frac{1}{(n+1)^{2}}\Bigg ((n+1)^{2}\left[\sum\limits_{k=2}^{n+1}(-1)^{k-1}\frac{1}{k}
\bigg(\sum\limits_{\sum\limits_{j=1}^{k}i_{j}=n+1}(\prod\limits_{j=1}^{k}a_{i_{j}})\bigg)\right]\\
&&+\sum\limits_{i=1}^{n}i^{2}a_{n+1-i}\left[a_{i}+\sum\limits_{k=2}^{i}(-1)^{k-1}\frac{1}{k}
\bigg(\sum\limits_{\sum\limits_{j=1}^{k}l_{j}=i}(\prod\limits_{j=1}^{k}a_{l_{j}})\bigg) \right] \Bigg)\\
&=&-\frac{1}{(n+1)^{2}}\left(\sum\limits_{i=1}^{n}i^{2}a_{n+1-i}a_{i}+
\sum\limits_{i=2}^{n+1}\sum\limits_{k=2}^{i}(-1)^{k-1}\frac{i^{2}}{k}
\left[\sum\limits_{\sum\limits_{j=1}^{k}l_{j}=i}(\prod\limits_{j=1}^{k}a_{l_{j}})a_{n+1-i}\right ]\right).
\end{array}$$

\end{proof}

To answer the question when $\frac{\partial^2}{{\partial z} \partial \overline w}\log K(z,w)=[K(z,w)]^{p}$ for some $p \in \mathbb{N}$ to hold, we need one more result.

\begin{lm}\label{DI1} For any $n \in \mathbb{N},$
$$\sum\limits_{k=1}^{n}(-1)^{k-1}\frac{1}{k}
\left(\sum\limits_{\sum\limits_{j=1}^{k}i_{j}=n}(i_{1}+1)(i_{2}+1)\cdots(i_{k}+1)\right)=\frac{2}{n}.$$
\end{lm}

\begin{proof}
Since $\log\left (\frac{1}{1-x}\right)^{2}=-2\log(1-x)=\log\left[1+\left(\frac{1}{(1-x)^{2}}-1\right)\right]$ for $|x|<1$,
$$\sum\limits_{n=1}^{\infty}\frac{2}{n}x^{n}=\sum\limits_{k=1}^{\infty}(-1)^{k-1}
\frac{1}{k}\left[\frac{1}{(1-x)^{2}}-1\right]^{k}=\sum\limits_{k=1}^{\infty}(-1)^{k-1}\frac{1}{k}
\left(\sum\limits_{n=2}^{\infty}nx^{n-1}\right)^{k}.$$
One now considers the coefficient of $x^n$ to get the result.
\end{proof}

\begin{thm}\label{tf6}
Let $K(z,w)=1+\sum\limits_{i=1}^{\infty}a_{i}z^{i}\overline w^{i}$ be a positive definite kernel on $\mathbb{D} \times \mathbb{D}$. For $p \in \mathbb{N},$ $\frac{\partial^2}{{\partial z} \partial \overline w}\log K(z,w)=[K(z,w)]^{p}$ if and only if $K(z,w)=\left(1-\frac{pz\overline w}{2}\right)^{-\frac{2}{p}}.$
\end{thm}

\begin{proof}
First, it is easy to see that for $K(z,w)=\left(1-\frac{pz\overline w}{2}\right)^{-\frac{2}{p}},$ $$\frac{\partial^2}{{\partial z} \partial \overline w}\log K(z,w)=-\frac{2}{p}\frac{\partial^2}{{\partial z} \partial \overline w}\log\left(1-\frac{pz\overline w}{2}\right)=\left(1-\frac{pz\overline w}{2}\right)^{-2}=[K(z,w)]^{p}.$$

For the other direction, let $L(z,w):= \left(K(z,w)\right)^{p}=\left(1+\sum\limits_{i=1}^{\infty}a_{i}z^{i}\overline w^{i}\right)^{p}=1+\sum\limits_{i=1}^{\infty}b_{i}z^{i}\overline w^{i}.$ One of the steps in the proof of Proposition \ref{tf5} showed that $$\frac{\partial^2}{{\partial z} \partial \overline w}\log
L(z,w)=\sum\limits_{n=1}^{\infty}n^{2}\left[\sum\limits_{k=1}^{n}(-1)^{k-1}\frac{1}{k}
\bigg(\sum\limits_{\sum\limits_{j=1}^{k}i_{j}=n}(\prod\limits_{j=1}^{k}b_{i_{j}})\bigg)\right]z^{n-1}\overline w^{n-1}.$$
Note that $\frac{\partial^2}{{\partial z} \partial \overline w}\log K(z,w)=[K(z,w)]^{p}$ is equivalent to $\frac{\partial^2}{{\partial z} \partial \overline w}\log L(z,w)=pL(z,w),$ that is,
$$\sum\limits_{n=1}^{\infty}n^{2}\left[\sum\limits_{k=1}^{n}(-1)^{k-1}\frac{1}{k}
\bigg(\sum\limits_{\sum\limits_{j=1}^{k}i_{j}=n}(\prod\limits_{j=1}^{k}b_{i_{j}})\bigg)\right]z^{n-1}\overline w^{n-1}=p+p\sum\limits_{i=1}^{\infty}b_{i}z^{i}\overline w^{i}.$$
Obviously, $b_{1}=p,b_{2}=\frac{3}{2^{2}}p^{2},$ and $b_{3}=\frac{4}{2^{3}}p^{3}.$ We will show that for all $i \geq 1$, $$b_{i}=\frac{i+1}{2^{i}}p^{i}.$$
This amounts to showing that the $b_{i}=\frac{i+1}{2^{i}}p^{i}$ for $1 \leq i \leq n$ satisfy
$$n^{2}\left[\sum\limits_{k=1}^{n}(-1)^{k-1}\frac{1}{k}
\left (\sum\limits_{\sum\limits_{j=1}^{k}i_{j}=n}(\prod\limits_{j=1}^{k}b_{i_{j}})\right)\right]=pb_{n-1},
$$
which is equivalent to
$$\begin{array}{lllll}
p^{n}\frac{n}{2^{n-1}}&=&n^{2}\left[\sum\limits_{k=1}^{n}(-1)^{k-1}\frac{1}{k}
\sum\limits_{\sum\limits_{j=1}^{k}i_{j}=n}\frac{(i_{1}+1)p^{i_{1}}}{2^{i_{1}}}
\frac{(i_{2}+1)p^{i_{2}}}{2^{i_{2}}}\cdots\frac{(i_{k}+1)p^{i_{k}}}{2^{i_{k}}}\right]\\
&=&\frac{n^{2}}{2^{n}}p^{n}\left[\sum\limits_{k=1}^{n}(-1)^{k-1}\frac{1}{k}
\sum\limits_{\sum\limits_{j=1}^{k}i_{j}=n}(i_{1}+1)(i_{2}+1)\cdots(i_{k}+1)\right].
\end{array}$$
By Lemma $\ref{DI1},$
$$\sum\limits_{k=1}^{n}(-1)^{k-1}\frac{1}{k}
\sum\limits_{\sum\limits_{j=1}^{k}i_{j}=n}(i_{1}+1)(i_{2}+1)\cdots(i_{k}+1)=\frac{2}{n},$$
and hence, $b_i=\frac{i+1}{2^i}p^i$ for all $i \geq 1$.
It then follows that $$L(z,w)=\left(K(z,w)\right)^{p}=1+\sum\limits_{n=1}^{\infty}\frac{n+1}{2^{n}}p^{n}z^{n}\overline w^{n}=\left(1-\frac{pz\overline w}{2}\right)^{-2},$$
and therefore, $$K(z,w)=\left(1-\frac{pz\overline w}{2}\right)^{-\frac{2}{p}}.$$

\end{proof}

\section{Similarity of Operators in $\mathcal{F}B_n(\Omega)$}
The following lemma states that the operator establishing the similarity between two operators in $\mathcal{F}B_n(\Omega)$ is of a special form:

\begin{lm}[\cite{JJKM}]\label{dp}
If $X$ is an invertible operator that intertwines operators in $\mathcal{F}B_n(\Omega)$, then $X$ and $X^{-1}$ are upper triangular.
\end{lm}

Recall that any homogeneous operator $T \in B_1(\mathbb{D})$ can be expressed as $M^*_z$, the adjoint of the operator of multiplication on the analytic function space $\mathcal{H}_{K_{\alpha}}$ with reproducing kernel $K_{\alpha}(z,w)=\frac{1}{(1-z\overline{w})^{\alpha}}$ for some $\alpha \in \mathbb{N}$ (see \cite{Misra} for details). At times, the similarity of operators in $\mathcal{F}B_2(\mathbb{D})$ can be determined exclusively by considering the related operators in $B_1(\mathbb{D})$ in the decomposition (\ref{1.1T}).

\begin{thm}\label{mainsim}
Let $T=\Big (\begin{smallmatrix}T_0 & S_{0,1} \\
0 & T_1 \\
\end{smallmatrix}\Big ), \text{ }S=\Big (\begin{smallmatrix}S^*_{0} & \widetilde{S}_{0,1}\\
0 & S^*_{1} \\
\end{smallmatrix}\Big ) \in \mathcal FB_2(\mathbb{D})$, where $S^*_i \sim_u (M^*_z, \mathcal{H}_{{K}_i})$ and ${K}_i(z,w)=\frac{1}{(1-z\bar{w})^{k_i}}$ for some $k_i \in \mathbb{N}$.
Suppose that the following statements hold:
\begin{enumerate}
\item[(1)] Each $T_i \in \mathcal{L}(\mathcal{H}_i)$ is a $k_i$-hypercontraction, and
\item[(2)] There exist $t_1(w)\in \text{ker}(T_1-w)$ and a function $\phi \in GL({H}^{\infty}(\mathbb{D}))$ such that for all $w\in\mathbb{D},$
$$|\phi(w)|^2\frac{\|S_{0,1}t_1(w)\|^2}{\|t_1(w)\|^2}=\frac{\|\widetilde{S}_{0,1}{K}_1(.,\overline{w})\|^2}{{K}_{1}(w,w)}.$$
\end{enumerate}
Then $T\sim_s S$ if and only if
$$\mathcal{K}_{S^*_1}-\mathcal{K}_{T_1} \leq \frac{\partial^{2}}{\partial\overline{w}\partial w} \psi,$$ for some bounded subharmonic function $\psi$ on $\mathbb{D}$.
\end{thm}

\begin{rem} Assumption (2) of Theorem \ref{mainsim} has a nice geometric interpretation. Note that for $\phi \in \text{Hol}(\mathbb{D})$,
$$\frac{\partial^{2}}{\partial\overline{w}\partial w}\log \left(|\phi(w)|^2\frac{\|S_{0,1}t_1(w)\|^2}{\|\widetilde S_{0,1}{K}_1(\cdot,\overline w)\|^2}\right) =\frac{\partial^{2}}{\partial\overline{w}\partial w}\log \frac{\|t_1(w)\|^2}{{K}_{1}(w,w)},$$
is equivalent to $$\mathcal{K}_{S_0^*}-\mathcal{K}_{T_0}=\mathcal{K}_{S^*_1}-\mathcal{K}_{T_1}.$$ Hence, one can state Theorem \ref{mainsim} with the condition
$$\mathcal{K}_{S^*_0}-\mathcal{K}_{T_0} \leq \frac{\partial^{2}}{\partial\overline{w}\partial w} \psi,$$
instead.
\end{rem}
\begin{proof}
Recall that for an operator $A$ that is an $n$-hypercontraction, the defect operators are defined for $1 \leq m \leq n$ by
$$D_{m, A}=\left( \sum\limits_{k=0}^{m}(-1)^{k}{m \choose k}A^{*k}A^k \right)^{\frac{1}{2}}.$$
We begin by defining the operators $V_0: \mathcal{H}_0 \rightarrow \mathcal{M}_0$ and $V_1: \mathcal{H}_1 \rightarrow \mathcal{M}_1$ by
$$V_ix=\sum\limits_{n=0}^{\infty}\frac{z^n}{\|z^n\|_i^2}\otimes D_{k_i,T_i}T^n_ix,$$ for $x\in {\mathcal H}_i,$ where $\mathcal{M}_i:=\overline{\text{ran }V_i}$ and $\|z^n\|_i$ denotes the norm of $z^n$ on the space $\mathcal{H}_{K_i}$. Then using J. Agler's result in \cite{Agler2}, we see that each $V_i$ is a unitary operator satisfying $V_iT_i=M^*_z|_{\mathcal{M}_i}V_i$.

Suppose that $t_0(w)\in \text{ker}(T_0-w)$ and $t_1(w)\in \text{ker}(T_1-w)$ are such that $S_{0,1}t_1(w)=t_0(w)$ for $w \in \mathbb{D}$. We then have
$$\begin{array}{lllll}V_0t_0(w)&=&\sum\limits_{n=0}^{\infty}\frac{z^n}{\|z^n\|_0^2}\otimes D_{k_0,T_0}T^n_0t_0(w)\\
&=&\sum\limits_{n=0}^{\infty}\frac{z^nw^n}{\|z^n\|_0^2}\otimes D_{k_0,T_0}t_0(w)\\
&=&{K}_0(z,\overline{w})\otimes D_{k_0,T_0}t_0(w),\\
\end{array}
$$
for $w\in\mathbb{D}$.
Analogously, one can show that $$V_1t_1(w)={K}_1(z,\overline{w})\otimes D_{k_1, T_1}t_1(w).$$

Now since $S\in \mathcal{F}B_2(\mathbb{D})$, $S^*_{0} \widetilde{S}_{0,1}=\widetilde{S}_{0,1}S^*_1$ and there exists a function $\chi \in \text{Hol}(\mathbb{D})$ such that
$${K}_{0}(\cdot,\overline{w})=\chi(w)\widetilde{S}_{0,1}{K}_{1}(\cdot,\overline{w}),$$ for all $w\in{\mathbb{D}}.$ If we set $$e(w):=\chi(w)D_{k_0,T_0}S_{0,1}t_1(w)\in\mathcal{H}_0,$$ then
$$\begin{array}{lllll}
\|S_{0,1}t_1(w)\|^2 &=& \|{K}_{0}(\cdot,\overline{w})\otimes D_{k_0,T_0}S_{0,1}t_1(w)\|^2\\
&=&\|\chi(w)\widetilde{S}_{0,1}{K}_1(\cdot,\overline{w})\otimes D_{k_0, T_0}S_{0,1}t_1(w)\|^2\\
&=&\|\widetilde{S}_{0,1}{K}_1(\cdot,\overline{w})\otimes e(w)\|^2\\
&=&\|\widetilde{S}_{0,1}{K}_1(\cdot,\overline{w})\|^2 \| e(w)\|^2.
\end{array}$$
Similarly, $$\|t_1(w)\|^2={K}_1(w,w) \|D_{k_1,T_1}t_1(w)\|^2,$$ and since
$$|\phi(w)|^2\frac{\|S_{0,1}t_1(w)\|^2}{\|t_1(w)\|^2}=\frac{\|\widetilde{S}_{0,1}{K}_1(\cdot,\overline {w})\|^2}{{K}_{1}(w,w)},$$
for some $\phi\in GL({H}^{\infty}(\mathbb{D}))$,
we have $$\|t_1(w)\|^2=|\phi(w)|^2{K}_1(w,w) \|e(w)\|^2.$$

By the Rigidity Theorem given in \cite{CD}, we next define the isometries $W_0$ and $W_1$ by
$$W_0S_{0,1}t_1(w):=\widetilde {S}_{0,1}{K}_1(\cdot, \overline{ w})\otimes e(w), \text{ and}$$
$$W_1t_1(w):=\phi(w){K}_1(\cdot,\overline{w})\otimes e(w),$$
for $w \in \mathbb{D}$.
Setting $\mathcal{N}_i =\overline{\text{ran }W_i}$, the isometries $W_i\in \mathcal{L}(\mathcal{H}_{i}, \mathcal{N}_i)$ become unitary operators and
\begin{equation} \label{equ}
\Big (\begin{matrix}W_0 & 0 \\
0 & W_1 \\
\end{matrix}\Big )\Big (\begin{matrix}T_0 & S_{0,1} \\
0 & T_1 \\
\end{matrix}\Big )\Big (\begin{matrix}W^*_0 & 0 \\
0 &W^*_1 \\
\end{matrix}\Big )=
\Big (\begin{matrix}W_0V_0^*M_z^*|_{\mathcal{M}_0}V_0W^*_0 & W_0S_{0,1}W_1^* \\
0 & W_1V_1^*M^*_z|_{\mathcal{M}_1}V_1W^*_1 \\
\end{matrix}\Big )
=
\Big (\begin{matrix}M^*_z|_{\mathcal{N}_0} & W_0S_{0,1}W^*_1 \\
0 & M^*_z|_{\mathcal{N}_1} \\
\end{matrix}\Big ).
\end{equation}
From this, we deduce that
$$
T_i \sim_{u} M^*_z |_{\mathcal{N}_i}.
$$
Moreover, by a result in \cite{Lin}, we have for $w \in \mathbb{D}$, $$\text{ker}(M^*_z|_{\mathcal{N}_0}-w)=\bigvee_{w \in \mathbb{D}}\widetilde{S}_{0,1}{K}_1(\cdot, \overline{w})\otimes e(w)  \text{ and }  \text{ker}(M^*_z|_{\mathcal{N}_1}-w)=\bigvee_{w \in \mathbb{D}}{K}_1(\cdot, \overline{w})\otimes e(w).$$

We now prove that the condition $\mathcal{K}_{S_1^*}-\mathcal{K}_{T_1} \leq \frac{\partial^{2}}{\partial\overline{w}\partial w}\psi$ is sufficient for the similarity between $T$ and $S$.
Since $T_i\sim_{u} M^*_z|_{\mathcal{N}_i}$, we have $$\mathcal{K}_{S^*_0}-\mathcal{K}_{T_0}=\mathcal{K}_{S^*_1}-\mathcal{K}_{T_1}=\mathcal{K}_{S^*_1}-\mathcal{K}_{M^*_z|_{\mathcal{
N}_1}}=\mathcal{K}_{S^*_1}-(\mathcal{K}_{S^*_1}+\mathcal{K}_{\mathcal{E}})=-\mathcal{K}_{\mathcal{E}} \leq\frac{\partial^{2}}{\partial\overline{w}\partial w}\psi,$$ where $\mathcal{E}$ denotes the bundle with fiber
$\mathcal{E}(w):=\bigvee e(w).$
Under this condition, it is shown in \cite{Kwon1} that there exist invertible operators $X_0\in \mathcal{L}(\mathcal{H}_{K_0},\mathcal{N}_0)$ and $X_1\in \mathcal{L}(\mathcal{H}_{K_1},\mathcal{N}_1)$
 such that $$X_iS^*_i=M^*_z|_{\mathcal{N}_i}X_i.$$
It then follows for every $w \in \mathbb{D}$ that
$$X_0{\widetilde S}_{0,1}{K}_1(\cdot,\overline w)=\lambda(w){\widetilde S}_{0,1}{K}_1(\cdot,\overline w)\otimes e(w),$$
and
$$X_1{K}_1(\cdot,\overline w)=\lambda(w)\phi(w){K}_1(\cdot,\overline w)\otimes e(w),$$
for some $\lambda(w) \in\text{Hol}(\mathbb{D})$.
Moreover,
$$\begin{array}{llll}
W_0S_{0,1}W^*_1X_1{K}_{1}(\cdot,\overline w)&=&W_0S_{0,1}W^*_1(\lambda(w)\phi(w){K}_1(\cdot,\overline w)\otimes e(w))\\
&=&W_0S_{0,1}(\lambda(w)t_1(w))\\
&=&\lambda(w){\widetilde S}_{0,1}{K}_1(\cdot,\overline w)\otimes e(w)\\
&=&X_0{\widetilde S}_{0,1}{K}_1(\cdot,\overline w),
\end{array}
$$
so that
$$\Big (\begin{matrix}X_0 & 0 \\
0 & X_1 \\
\end{matrix}\Big )\Big (\begin{matrix}S^*_{0} & \widetilde S_{0,1} \\
0 & S^*_{1} \\
\end{matrix}\Big ) =
\Big (\begin{matrix}M^*_z|_{\mathcal{N}_0} & W_0S_{0,1}W^*_1 \\
0 & M^*_z|_{\mathcal{N}_1} \\
\end{matrix}\Big )\Big (\begin{matrix}X_0 & 0 \\
0 & X_1 \\
\end{matrix}\Big ).$$
Combining this result with (\ref{equ}), we finally conclude that $T\sim_s S$.

For the necessity, assume that $XT=SX$ for some invertible operator $X$. Then by Lemma \ref{dp},
$X=\Big (\begin{smallmatrix}X_0 & X_{0,1} \\
0 & X_1 \\
\end{smallmatrix}\Big )$ and since $X^{-1}$ is also upper-triangular, both $X_0$ and $X_1$ are invertible. Moreover,
$X_iT_i=S^*_iX_i$. Now, since $T_1$ is a $k_1$-hypercontraction,  by \cite{Kwon2}, there exists a bounded subharmonic function $\psi$ defined on $\mathbb{D}$ such that
$$\mathcal{K}_{S^*_1}-\mathcal{K}_{T_1} \leq \frac{\partial^{2}}{\partial\overline{w}\partial w} \psi.$$

\end{proof}

The following example shows that the condition $\phi \in GL(H^{\infty}(\mathbb{D}))$ in Theorem \ref{mainsim} is not an unreasonable assumption:
\begin{ex} Let $S=\Big (\begin{smallmatrix}S^*_{0} & \widetilde S_{0,1}\\
0 & S^*_{1} \\
\end{smallmatrix}\Big ) \in \mathcal FB_2(\mathbb{D})$ and let $S_{\phi}=\Big (\begin{smallmatrix}S^*_{0} &\phi(S^*_0) \widetilde S_{0,1}\\
0 & S^*_{1} \\
\end{smallmatrix}\Big )$ for some $\phi\in {H}^{\infty}(\mathbb{D})$ (note that $S_{\phi}\in {\mathcal FB_2(\mathbb{D})}$ as well). Suppose that $S^*_i \sim_u (M^*_z, \mathcal{H}_{{K}_i})$ with the reproducing kernel given by ${K}_i(z,w)=\frac{1}{(1-z\bar{w})^{k_i}}$ for some $k_i \in \mathbb{N}$. Note that the operators $S^*_0$ and $S^*_1$ can then be viewed as weighted shift operators with weight sequences $\left \{ \sqrt{\frac{n+1}{n+k_i}}\right\}^{\infty}_{n=0}.$

It is shown in \cite{JJK} that if $\lim\limits_{m\rightarrow \infty}m\frac{\prod\limits^{m}_{n=0}\sqrt{\frac{n+1}{n+k_1}}}{\prod\limits^{m}_{n=0}\sqrt{\frac{n+1}{n+k_0}}}=\infty,$ then an invertible operator $X$ that intertwines $S$ and $S_{\phi}$ should be diagonal. Since Stirling's formula gives $$\prod\limits^{m}_{n=0}\sqrt{\frac{n+1}{n+k_0}}\sim O(m^{\frac{1-k_0}{2}}) \text{ and } \prod\limits^{m}_{n=0}\sqrt{\frac{n+1}{n+k_1}}\sim O(m^{\frac{1-k_1}{2}}),$$ this is true when $k_1-k_0>2$. Then,
 \begin{equation}\nonumber
S\sim_s S_{\phi}\Leftrightarrow
 \left\{\begin{array}{lll}
X_0S^*_0=S^*_0X_0, \\
X_1S^*_1=S^*_1X_1,\\
X_{0} \widetilde{S}_{0,1}=\phi(S^*_0) \widetilde{S}_{0,1}X_1,\\
   \end{array}
\right.
\end{equation}
for some invertible operators $X_0\in \mathcal{L}(\mathcal{H}_{K_0})$ and $X_1\in \mathcal{L}(\mathcal{H}_{K_1}).$ Since $\{S^*_i\}^{\prime}={H}^{\infty}(\mathbb{D})$,  there exist
$\phi_0, \phi_1 \in GL({H}^{\infty}(\mathbb{D}))$ such that $X_i=\phi_i(S^*_i)$. Then by the equation $X_{0}\widetilde{S}_{0,1}=\phi(S^*_0) \widetilde {S}_{0,1}X_1$, we have $$\phi_0(S^*_0)\widetilde{S}_{0,1}=\phi(S^*_0)\phi_1(S^*_0)\widetilde{S}_{0,1}.$$
Since it is known that $\widetilde{S}_{0,1}$ has dense range (see \cite{JJKM}), it follows that $\phi_0(S^*_0)=\phi(S^*_0)\phi_1(S^*_0)$, and therefore, $\phi\in  GL({H}^{\infty}(\mathbb{D}))$.

\end{ex}

Once an additional intertwining condition is imposed, Theorem \ref{mainsim} can be generalized to operators in the class $\mathcal{F}B_n(\mathbb{D})$:
\begin{thm}\label{mainthm} Let $T=\begin{tiny}\begin{pmatrix}
T_{0} & S_{0,1} & S_{0,2}&\cdots&S_{0,n-2}&S_{0,n-1}\\
0 &T_{1}&S_{1,2}&\cdots&S_{1,n-2}&S_{1,n-1} \\
0 &0&T_{2}&\cdots&S_{2,n-2}&S_{2,n-1} \\
\vdots&\vdots&\vdots&\ddots&\vdots&\vdots\\
0&0&0&\cdots&T_{n-2}&S_{n-2,n-1}\\
0&0&0&\cdots&0&T_{n-1}\\
\end{pmatrix}\end{tiny} \text{and }S=\begin{tiny}\begin{pmatrix}
S^*_{0} & {\widetilde S}_{0,1} & {\widetilde S}_{0,2}&\cdots&{\widetilde S}_{0,n-2}&{\widetilde S}_{0,n-1}\\
0 &S^*_{1}&{\widetilde S}_{1,2}&\cdots& {\widetilde S}_{1,n-2}& {\widetilde S}_{1,n-1} \\
0 &0 &S^*_{2}&\cdots& {\widetilde S}_{2,n-2}& {\widetilde S}_{2,n-1} \\
\vdots&\vdots&\vdots&\ddots &\vdots&\vdots\\
0&0&0&\cdots&S^*_{n-2}&{\widetilde S}_{n-2,n-1}\\
0&0&0&\cdots&0&S^*_{n-1}\\
\end{pmatrix}\end{tiny}$ both be in $\mathcal{F}B_n(\mathbb{D})$,  where $S^*_i=(M^*_z, \mathcal{H}_{{K}_i})$ and  ${K}_i(z,w)=\frac{1}{(1-z\bar{w})^{k_i}}$ for some $k_i \in \mathbb{N}$ and for all $0 \leq i \leq n-1$.
Suppose that the following conditions hold:
\begin{enumerate}
\item[(1)] Each $T_i \in \mathcal{L}(\mathcal{H}_i)$ is a $k_i$-hypercontraction for $0 \leq i \leq n-1,$
\item[(2)] There exist functions $\{\phi_i\}^{n-1}_{i=0} \subset GL({H}^{\infty}(\mathbb{D}))$ such that for all $0 \leq i < j \leq n-1$ and for all $w \in \mathbb{D}$,
$$\prod\limits_{k=i}^{j-1}|\phi_k(w)|^2\frac{|\langle S_{i,j}t_j(w), t_i(w)\rangle|}{\|t_{j}(w)\|^2}=\frac{|\langle {\widetilde S}_{i,j} {\widetilde {K}}_j(w), {\widetilde{ K}}_i(w)\rangle|}{\|{\widetilde {K}}_{j}(w)\|^2},$$  where $t_{n-1}(w) \in \text{ker }(T_{n-1}-w)$, $\widetilde{K}_{n-1}(w)={K}_{n-1}(\cdot,\overline{w}),$ and the other terms are inductively defined as $t_{n-i}(w)=S_{n-i, n-i+1}t_{n-i+1}(w)$ and  ${\widetilde{K}}_{n-i}(w)={\widetilde S}_{n-i,n-i+1}{\widetilde {K}}_{n-i+1
}(w)$ for $2 \leq i \leq n,$ and
\item[(3)] $T_iS_{i, j}=S_{i, j}T_j$ and ${S^*_i} \widetilde{S}_{i, j}= \widetilde{S}_{i, j}{S^*_j}$ for all $0 \leq i < j \leq n-1$.
\end{enumerate}

Then $T\sim_s S$ if and only if
$$\mathcal{K}_{S^*_{n-1}}-\mathcal{K}_{T_{n-1}} \leq\frac{\partial^{2}}{\partial\overline{w}\partial w}\psi,$$ for some bounded subharmonic function $\psi$ defined on $\mathbb{D}$.

\end{thm}

\begin{proof} As in the proof of Theorem  \ref{mainsim}, there exists a holomorphic Hermitian vector bundle $\mathcal{E}$ over $\mathbb{D}$ with fiber $\mathcal{E}(w)=\bigvee e(w)$ such that for $0 \leq i \leq n-2$,
$$\|t_i(w)\|^2=\|S_{i,i+1}t_{i+1}(w)\|^2 =\|\widetilde{S}_{i,i+1}\widetilde{{K}}_{i+1}(w)\|^2 \| e(w)\|^2=\|\widetilde{K}_{i}(w)\|^2\| e(w)\|^2,$$
where $t_i(w)\in \text{ker}(T_i-w)$, $t_{i+1}(w)\in \text{ker}(T_{i+1}-w),$ and $S_{i,i+1}t_{i+1}(w)=t_i(w)$ for $w\in\mathbb{D}$.
Now let $j=i+1$ in assumption (2) to obtain  $$|\phi_i(w)|^2\frac{\|t_i(w)\|^2}{\|t_{i+1}(w)\|^2}=\frac{\|{\widetilde {K}}_i(w)\|^2}{\|{\widetilde { K}}_{i+1}(w)\|^2},$$ from which it follows for $1 \leq i \leq n-1$ that $$\|t_i(w)\|^2=\prod\limits_{k=0}^{i-1} |\phi_k(w)|^2\|{\widetilde{ K}}_{i}(w)\|^2\| e(w)\|^2.$$

We next define the isometries $W_i$ as $W_0t_0(w)=\widetilde{K}_0(w) \otimes e(w)$ and for $1 \leq i \leq n-1$,$$W_it_i(w)=\prod\limits_{k=0}^{i-1} \phi_k(w){\widetilde{ K}}_{i}(w)\otimes e(w).$$ Then
$$\begin{tiny}\begin{pmatrix}
T_{0} & S_{0,1} & S_{0,2}&\cdots&S_{0,n-2}&S_{0,n-1}\\
0 &T_{1}&S_{1,2}&\cdots&S_{1,n-2}&S_{1,n-1} \\
0 &0&T_{2}&\cdots&S_{2,n-2}&S_{2,n-1} \\
\vdots&\vdots&\vdots&\ddots&\vdots&\vdots\\
0&0&0&\cdots&T_{n-2}&S_{n-2,n-1}\\
0&0&0&\cdots&0&T_{n-1}\\
\end{pmatrix}\sim_{u}\begin{pmatrix}
M^*_z|_{\mathcal{N}_0} & W_0{S}_{0,1}W^*_1 & W_0{ S}_{0,2}W^*_2&\cdots&W_0{S}_{0,n-2}W^*_{n-2}&W_0{S}_{0,n-1}W^*_{n-1}\\
0 &M^*_z|_{\mathcal{N}_1}&W_1{S}_{1,2}W^*_2&\cdots& W_1{S}_{1,n-2}W^*_{n-2}& W_1{S}_{1,n-1}W^*_{n-1} \\
0 &0&M^*_z|_{\mathcal{N}_2}&\cdots& W_2{S}_{2,n-2}W^*_{n-2}& W_2{S}_{2,n-1}W^*_{n-1} \\
\vdots&\vdots&\vdots&\ddots&\vdots&\vdots\\
0&0&0&\cdots&M^*_z|_{\mathcal{N}_{n-2}}&W_{n-2}{ S}_{n-2,n-1}W^*_{n-1}\\
0&0&0&\cdots&0&M^*_z|_{\mathcal{N}_{n-1}}\\
\end{pmatrix}\end{tiny},$$
 $\mathcal{N}_i=\overline{\text{ran}W_i}$ for $0\leq i\leq n-1$. Proceeding again as in the proof of Theorem  \ref{mainsim}, there exist invertible operators $X_i\in \mathcal{L}(\mathcal{H}_{K_i}, \mathcal{N}_i)$
for $0\leq i\leq n-1$ such that $$X_iS^*_i=M^*_z|_{\mathcal{N}_i}X_i.$$
Furthermore, there exists some $\lambda(w) \in \text{Hol}(\mathbb{D})$ satisfying $$X_0{\widetilde S}_{0,j}\widetilde{K}_j(w)=\lambda(w){\widetilde S}_{0,j}\widetilde{K}_j(w)\otimes e(w),$$
and for $1\leq j\leq n-1$,
$$X_j\widetilde{K}_j(w)=\lambda(w)\prod\limits^{j-1}_{k=0}\phi_k(w)\widetilde{K}_j(w)\otimes e(w).$$ It can also be checked through direct calculation that for $0 \leq i \leq n-2$,
$$X_i\widetilde{S}_{i, i+1}=W_iS_{i, i+1}W^*_{i+1}X_{i+1}.$$

To prove that $T$ is similar to $S$, we need only check that for $0 \leq i < j \leq n-1$,
 $$X_{i}\widetilde{S}_{i,j}= W_{i}S_{i,j}W^*_jX_j.$$ Note that since  $T_iS_{i,j}=S_{i,j}T_{j}$ and  $S^*_i{\widetilde S}_{i,j}={\widetilde S}_{i,j}S^*_{j}$, there exist functions $\psi_{i,j}, \widetilde{\psi}_{i,j} \in \text{Hol}(\mathbb{D})$ such that $S_{i,j}t_{j}=\psi_{i,j}t_i$ and $\widetilde{S}_{i,j}{\widetilde{ K}}_{j}=\widetilde{\psi}_{i,j}{\widetilde{K}}_{i}.$ Then for $1 \leq i < j \leq n-1$,
 $$X_{i}{\widetilde S}_{i,j}{\widetilde{ K}}_{j}(w)=X_i(\widetilde{\psi}_{i,j}(w){\widetilde{ K}}_{i}(w))
=\lambda(w)\widetilde{\psi}_{i,j}(w)\prod\limits_{k=0}^{i-1} \phi_k(w){\widetilde{K}}_{i}(w)\otimes e(w)$$
and
$$\begin{array}{llll}
W_{i}S_{i,j}W^*_{j}X_{j}{\widetilde{ K}}_{j}(w)&=&W_{i}S_{i,j}W^*_{j}\left(\lambda(w)\prod\limits_{k=0}^{j-1} \phi_k(w){\widetilde{K}}_{j}(w)\otimes e(w)\right)\\
&=&\lambda(w)W_{i}S_{i,j}t_{j}(w)\\
&=&\lambda(w)W_{i}(\psi_{i,j}(w)t_i(w))\\
&=&\lambda(w)\psi_{i,j}(w)\prod\limits_{k=0}^{i-1} \phi_k(w){\widetilde{K}}_{i}(w)\otimes e(w).
\end{array}$$ In addition,  for $0 < j \leq n-1$, $$X_0\widetilde{S}_{0,j}\widetilde{K}_j(w)=\lambda(w)\widetilde{\psi}_{0,j}(w) \widetilde{K}_0(w) \otimes e(w),$$ and $$W_0S_{0,j}W^*_jX_j\widetilde{K}_j(w)=\lambda(w){\psi}_{0,j}(w) \widetilde{K}_0(w) \otimes e(w).$$ It now remains to prove that for $0 \leq i < j \leq n-1$, $\psi_{i,j}=\widetilde{\psi}_{i,j}.$ Note that $$\prod_{k=i}^{j-1}|\phi_k(w)|^2\frac{\|t_i(w)\|^2}{\|t_{j}(w)\|^2}=\frac{\| {\widetilde{K}}_i(w)\|^2}{\|{\widetilde{K}}_{j}(w)\|^2}$$
implies that $|\psi_{i,j}|=|\widetilde{\psi}_{i,j}|$. Since $\psi_{i,j}, \widetilde{\psi}_{i,j} \in \text{Hol}(\mathbb{D})$, we conclude that $\psi_{i,j}=\widetilde{\psi}_{i,j}$. This finishes the proof of the sufficiency. The proof of the necessity parallels that of Theorem \ref{mainsim}.

\end{proof}

\section{ Operator Theoretic Realization and Similarity}
The realization of Hermitian holomorphic bundles gives natural operations between Cowen-Douglas operators. A related question then is the following: Given a Hermitian holomorphic bundle $E$, when can
one find a Cowen-Douglas operator $T$ such that $\mathcal{E}_T=E?$ It is known that at least for ${E}=\mathcal{E}_{T_1}\otimes
\mathcal{E}_{T_2}$ with $T_1\in B_n(\Omega)$ and $T_2\in
B_m(\Omega)$, such a Cowen-Douglas operator $T$ exists. In
\cite{Lin},  Q. Lin proved the existence of a Cowen-Douglas operator
``$T_1*T_2$'' defined on the space $\bigvee \limits_{w \in \Omega}
\ker(T_1-w) \otimes \ker(T_2-w)$ such that
$\mathcal{E}_{T_1*T_2}=\mathcal{E}_{T_1}\otimes \mathcal{E}_{T_2}.$
However, for tensor products of holomorphic bundles in general, the
answer to this question is still unknown. For example,  we can consider the following question:

{\bf Question}\,\,  For any
Hermitian holomorphic bundle $\mathcal{E}$ with rank $m$ and a
Cowen-Douglas operator $T\in B_n(\Omega)$, does there exists an
operator $S$ such that $\mathcal{E}_S=\mathcal{E}_T\otimes \mathcal{E}$?

Note that the
problem is also related to the similarity of Cowen-Douglas
operators. According to the work initiated by the second author and
S. Treil, an operator model theorem plays a key role in the
similarity problem. If $T_1$ is a Cowen-Douglas operator of index
one, an operator $T$ similar to ${T^{n}_1}$ is assumed to have a
holomorphic bundle $\mathcal{E}_{T}$ with a tensor product
structure. When  $T_1$ is $M_z^*,$ the adjoint of the multiplication
operator on a weighted Bergman space, this kind of geometric structure
of the operator $T$ can be naturally obtained for $T$ that is an
$n$-hypercontraction.  In this case, $\mathcal{E}_{T}$ is unitarily
equivalent to $\mathcal{E}_{T_1}\otimes \mathcal{E}$ for some
holomorphic bundle $\mathcal{E}$. Since $T$ is similar to $T_1$,
this bundle $\mathcal{E}$ cannot have any Cowen-Douglas operator
theoretical realization. This means that $\mathcal{E}_{T}$ cannot be
equal to $\mathcal{E}_{T_1}\otimes \mathcal{E}_{T_2}$ for any
Cowen-Douglas operator $T_2$.  Now, when $T_1$ is a Cowen-Douglas
operator with index $n$, the problem of determining similarity does
not have a clear solution. To give a sufficient condition for the
similarity of irreducible Cowen-Douglas operators without an
operator model theorem, we need the following result on operator
theoretical realization. This theorem  also gives a positive answer to the above question in a special case.

Denote by $\text{Hol}(\Omega, \mathbb{C}^m)$ the space of all $\mathbb{C}^m$-valued holomorphic functions defined on a domain $\Omega$.  Let $T \in B_n(\Omega)$ be such that $T\sim_{u} (M^*_z, \mathcal{H}_{{K}})$, where ${K}(z, w)=({K}_{i,j}(z,w))_{m\times m}$ and $\mathcal{H}_{{K}}\subseteq
\text{Hol}(\Omega, \mathbb{C}^m)$.

\begin{thm}\label{Lin2}Let $e_i(w), 1\leq i\leq n,$ be n holomorphic functions on $\Omega$ and  let $$e(w):=(e_1(w),e_2(w),
\cdots, e_m(w))\in \mathbb{C}^{m}, \quad w\in \Omega.$$  If $\mathcal{E}$
is a line bundle with
$$\mathcal{E}(w)=\bigvee\limits_{w \in \Omega}\{e(w)\},$$ then for any operator $T\in B_n(\Omega)$,  there
exists an operator $S$ such that
$\mathcal{E}_{S}=\mathcal{E}_{T}\otimes \mathcal{E}.$

\end{thm}

\begin{proof} Let $\{\sigma_i\}_{i=1}^m$ be an orthonormal basis for $\mathbb{C}^m$.  Then for $w \in \Omega$, $$\text{ker} (T-w)=\bigvee _{1 \leq i \leq n} {K}(\cdot,\overline w)\sigma_i.$$ Now set $$\mathcal{M}:=\bigvee\limits_{w\in \Omega}\{{K}(\cdot,\overline w)\sigma_i\otimes e(w), 1 \leq i \leq n\}, $$ which is an invariant subspace of $T\otimes I_{m}$, and let $$S:=(T\otimes I_{m})|_{\mathcal{M}}.$$ We need only prove that for $w \in \Omega$, $$\text{ker}(S-w)=\bigvee_{1 \leq i \leq n} {K}(\cdot,\overline w)\sigma_i\otimes e(w)=(\mathcal{E}_{T}\otimes \mathcal{E})(w).$$
Note that for any ${K}(\cdot,\overline w)\sigma_i\otimes e(w) \in\mathcal{M}$, we have
\begin{equation*}
S({K}(\cdot,\overline w)\sigma_i\otimes e(w))=(T\otimes I_{m})({K}(\cdot,\overline w)\sigma_i\otimes e(w))
=T({K}(\cdot,\overline w)\sigma_i)\otimes e(w)
=w {K}(\cdot,\overline w)\sigma_i\otimes e(w),
\end{equation*}
and hence, $(\mathcal{E}_{T}\otimes \mathcal{E})(w)\subseteq\text{ker}(S-w)$ for  $w\in\Omega.$  For the converse, we first consider the following lemma:

\begin{lm}
The orthogonal complement $\mathcal{M}^{\perp}$ of $\mathcal{M}$ can be represented as $$\mathcal{M}^{\perp}=\left \{(x_1,x_2,\cdots,x_m)\in \bigoplus \limits_{i=1}^n \mathcal {H}_{{K}} : \sum\limits_{j=1}^m\overline{e_j(w)}x^i_j(\overline w)=0 \text{ for }1 \leq i \leq n\right \},$$ where $x_j=(x^1_j,x^2_j,\cdots,x^n_j)^T\in\text{Hol}(\Omega, \mathbb{C}^n)$.
\end{lm}

\begin{proof}

Note that for $w \in \Omega$, $${K}(\cdot,\overline w)\sigma_i\otimes e(w)=
\left ({K}(\cdot,\overline w)\sigma_i e_1(w), {K}(\cdot,\overline w)\sigma_i e_2(w), \cdots, {K}(\cdot,\overline w)\sigma_i e_m(w)\right).$$
It then follows that $\mathcal{M}\subseteq \bigoplus\limits^{n}_{i=1}{\mathcal H}_{K},$ and therefore for any $x=(x_1,x_2,\cdots,x_m)\in M^{\perp}$, $$x_j=(x^1_j,x^2_j,\cdots,x^n_j)^T\in \text{Hol}(\Omega, \mathbb{C}^n).$$  Moreover, we also have
\begin{equation*}\begin{array}{llll}
\bigg \langle x, {K}(\cdot,\overline w)\sigma_i \bigg \rangle &=& \bigg \langle (x_1,x_2,\cdots,x_m), \left ({K}(\cdot,\overline w)\sigma_i e_1(w),  \cdots, {K}(\cdot,\overline w)\sigma_i e_m(w) \right) \bigg \rangle\\
&=& \sum\limits_{j=1}^m \bigg \langle  \begin{pmatrix}x^1_j \\
\vdots \\ x^n_j \\
\end{pmatrix} , {K}(\cdot,\overline w)\sigma_i e_j(w)  \bigg \rangle \\
&=&\sum\limits_{j=1}^m\overline{e_j(w)}x^i_j(\overline w)\\
&=&0.
\end{array}
\end{equation*}
\end{proof}
For any  $t=(t_1,t_2,\cdots,t_m)\in \text{ker}(S-w)$, we have $t_i\in \text{ker}(T-w)$. Then there exist functions $\{\alpha^i_j\}_{i=1} ^n \subseteq \text{Hol}(\Omega)$ such that for $1 \leq i \leq n$, $$t_j=\sum\limits_{i=1}^n\alpha^i_j(w){K}(\cdot,\overline w)\sigma_i.$$ It follows that for any $x=(x_1,x_2,\cdots,x_m)\in \mathcal{M}^{\perp}$,
\begin{equation*}\begin{array}{llll}
\langle x, t\rangle &=&\bigg \langle (x_1,x_2,\cdots,x_m), (t_1,t_2,\cdots,t_m) \bigg \rangle\\[4pt]
&=&\bigg \langle (x_1,x_2,\cdots,x_m), \left (\sum\limits_{i=1}^n\alpha^i_1(w){K}(\cdot,\overline w)\sigma_1,\sum\limits_{i=1}^n\alpha^i_2(w){K}(\cdot,\overline w)\sigma_2,\cdots,\sum\limits_{i=1}^n\alpha^i_m(w){K}(\cdot,\overline w)\sigma_m \right) \bigg \rangle \\
&=&\sum\limits_{j=1}^m\sum\limits_{i=1}^n\overline{\alpha^i_j(w)}x_j^i(\overline w)\\
&=&0.
\end{array}
\end{equation*}

In particular, if one sets $x^j_1=x^j_2=\cdots=x^j_n=0,$ then for any $j\neq i$,
$$\sum\limits_{j=1}^m\overline{\alpha^i_j(w)}x^i_j(\overline w)=0.
$$
Recall from before that the $x_j^i$ also satisfy $\sum\limits_{j=1}^m\overline{e_j(w)}x^i_j(\overline w)=0.$
Hence for any $i_1$ and $i_2$, if one sets $x^{i_1}_j(\overline w)=-\overline{e_{i_2}(w)}$, $x^{i_2}_j(\overline w)=\overline{e_{i_1}(w)}$, and $x^i_j(\overline w)=0$ for $i$ different from $i_1$ and $i_2$, then $x\in \mathcal{M}^{\perp}$. Moreover,  $\overline{\alpha^{i_1}_j(w)e_{i_2}(w)}=\overline{\alpha^{i_2}_j(w)e_{i_1}(w)}$.  Without loss of generality, we assume that for all $w \in \Omega$ and $1 \leq i \leq m$, $e_i(w)\neq 0.$ Then for each $1 \leq i \leq n$, there exist $m$ holomorphic functions $$\frac{\alpha^i_1}{e_1}=\frac{\alpha^i_2}{e_2}=\cdots=\frac{\alpha^i_m}{e_m}$$ that are equal to one another.
Thus,
 \begin{equation*}\begin{array}{llll}
 (t_1,t_2,\cdots,t_m)&=&\left (\sum\limits_{i=1}^n\alpha_1^i(w){K}(\cdot,\overline w)\sigma_i, \sum\limits_{i=1}^n\alpha_2^i(w){K}(\cdot,\overline w)\sigma_i,\cdots, \sum\limits_{i=1}^n\alpha_m^i(w){K}(\cdot,\overline w)\sigma_i \right)\\
 &=&\sum \limits_{i=1}^n \left (\alpha^i_1(w){K}(\cdot,\overline w)\sigma_i,\cdots,\alpha^i_m(w){K}(\cdot,\overline w)\sigma_i\right )\\
 &=&\sum\limits_{i=1}^n{K}(\cdot,\overline w)\sigma_i\otimes (\alpha^i_1(w),\alpha^i_2(w),\cdots,\alpha^i_m(w))\\
 &=&\sum\limits_{i=1}^nk_i(w){K}(\cdot,\overline w)\sigma_i\otimes (e_1(w),e_2(w),\cdots,e_m(w))\\
  &=&\sum\limits_{i=1}^nk_i(w){K}(\cdot,\overline w)\sigma_i\otimes e(w),\\
 \end{array}
 \end{equation*}
 where $k_i:=\frac{\alpha^i_1}{e_1}.$
 This means that for $w \in \Omega$, $\text{ker}(S-w)\subseteq (\mathcal{E}_T\otimes \mathcal{E})(w)$ and the proof is complete.
 \end{proof}

Before moving onto the next theorem, we need a few more notations and lemmas.  Let $T \in B_n(\Omega)$ be an operator defined on $\mathcal{H}$ such that for $w \in \Omega$,
$\text{ker}(T-w)=\bigvee \limits_{i=1}^n e_i(w)$ for some holomorphic $e_i(w)$.
If we define an operator-valued function $\alpha : \Omega\rightarrow
{\mathcal L}(\mathbb{C}^n, {\mathcal H})$ as
$$\alpha(w)(w_1,w_2,\cdots,w_n):=\sum\limits_{i=1}^nw_i e_i(w),$$ then the Gram matrix $h$ is related to $\alpha$ by
$$
h(w) = \alpha(w)^* \alpha(w),
$$
for $w \in \Omega$. Then ${P}_{\text{ker}(T - w)}$, the projection from $\mathcal{H}$ onto $\ker(T-w)$, can be written as
$$
P_{\text{ker}(T - w)} = \alpha(w)h^{-1}(w) \alpha^*(w).
$$
When no confusion arises, we will also use the notation ${P}(w)$ to denote ${P}_{\text{ker} (T - w)}$. This projection formula first appeared in the work of R. Curto and N. Salinas in \cite{CS}. See also the references \cite{JI} and \cite{JS} for further generalization. In particular, we mention below the result due to the first author given in \cite{JI}. We first start with some relevant definitions and results.
\begin{Definition}For a unital $C^*$-algebra $\mathfrak{U}$, $p$ is called a projection (or an orthogonal projection) in $\mathfrak{U}$ whenever
 $p^2=p=p^*$. The set of all projections in $\mathfrak{U}$ is called the Grassmann manifold of ${\mathfrak{U}}$ and is denoted by ${\mathcal P}(\mathfrak{U})$. For a connected open set $\Omega\subset \mathbb{C}$, $P:\Omega \rightarrow\mathcal{P}(\mathfrak{U})$ is said to be a holomorphic curve on $\mathcal{P}({\mathfrak{U}})$ if it is a real-analytic $\mathfrak{U}$-valued map satisfying $\frac{\partial}{\partial \overline{w}}PP=0$.
 \end{Definition}

 \begin{lm}[\cite{MS1}] \label{MSL} For a holomorphic curve $P$ on $\mathcal{P}(\mathfrak{U})$, we have for all positive integers $I$ and $J$,
$$\frac{\partial^{J}}{\partial^{J} \overline{w}}PP=P\frac{\partial^{I}}{\partial^{I} w}P=0.$$

 \end{lm}

 \begin{Definition}\label{holomorphic}
Let $\Omega \subset \mathbb{C}$ be a connected open set and suppose
${\mathfrak{U}}$ is a unital $C^*$-algebra. Given a  holomorhic curve $P:\Omega\rightarrow
{\mathcal P(\mathfrak{U})}$, the curvature and the corresponding covariant derivatives of the holomorphic curve $P$, denoted $\mathcal{K}_{i,j}(P)$ for $i, j \geq 0$, are defined as
 \begin{align*}  {\mathcal K}(P):={\mathcal K}_{0,0}(P)=\frac{\partial}{\partial \overline{w}}P\frac{\partial}{\partial w} P,\\
 {\mathcal K}_{i+1,j}(P):=P\frac{\partial}{\partial w}({\mathcal K}_{i,j}(P)), \text{ and }\\
{\mathcal K}_{i,j+1}(P):=\frac{\partial}{\partial \overline{w}}({\mathcal
K}_{i,j}(P))P.\end{align*}
\end{Definition}

\begin{lm}[\cite{JI}] \label{Jlem} Let $P(w)=\alpha(w)\left (\alpha^{*}(w)\alpha(w)\right )^{-1}\alpha^{*}(w)$ be the projection onto $\text{ker}(T-w)$ defined above.
Then the curvature and its covariant derivatives ${\mathcal K}_{i,j}(P):\Omega\rightarrow {\mathcal {\mathcal L}({\mathcal H})}$ for $0 \leq i,j \leq n$, satisfy the identity
$${\mathcal K}_{i,j}(P)(w)=\alpha(w)(-{\mathcal{K}}_{T,z^i{\overline
 z}^j}(w))h^{-1}(w)\alpha^*(w),$$
 for all $w \in \Omega$.
\end{lm}
Based on these lemmas, we can prove the following result:

 \begin{thm}\label{cf} Let ${\mathcal E}_1$ and ${\mathcal E}_2$ be Hermitian holomorphic vector bundles over $\Omega$. Set ${\mathcal H}_i=\bigvee\limits_{w\in \Omega}{\mathcal{E}_i(w)}$. If the $P_i(w)$ denote the projection from ${\mathcal H}_i$ onto $\mathcal{E}_i(w)$, then
 $$ {\mathcal K}_{i,j}(P_1\otimes P_2)={\mathcal K}_{i,j}(P_1)\otimes P_2+P_1\otimes {\mathcal K}_{i,j}(P_2).$$
 \end{thm}

\begin{proof}
 We prove by induction on $i$ and $j$ and consider the case  $i=j=0$ first.  Notice that
\begin{equation*}\begin{array}{llll}
{\mathcal K}(P_1\otimes P_2)&=&\frac{\partial}{\partial \overline{w}}(P_1\otimes P_2) \frac{\partial}{\partial w}(P_1\otimes P_2)\\[4pt]
&=& (\frac{\partial}{\partial \overline{w}}P_1 \otimes P_2 +P_1\otimes \frac{\partial}{\partial \overline{w}} P_2) (\frac{\partial}{\partial w}P_1 \otimes P_2 +P_1\otimes  \frac{\partial}{\partial w}P_2)\\[4pt]
&=& (\frac{\partial}{\partial \overline{w}}P_1\frac{\partial}{\partial w} P_1\otimes P_2+P_1\otimes \frac{\partial}{\partial \overline{w}} P_2\frac{\partial}{\partial w}P_2 +\frac{\partial}{\partial \overline{w}}P_1 P_1\otimes P_2 \frac{\partial}{\partial w}P_2+
P_1\frac{\partial}{\partial \overline{w}}P_1\otimes \frac{\partial}{\partial \overline{w}}P_2P_2).
\end{array}
\end{equation*}
By Lemma \ref{MSL}, $\frac{\partial}{\partial \overline{w}} P_1 P_1=\frac{\partial}{\partial \overline{w}}P_2P_2=0$ and hence, $${\mathcal K}(P_1\otimes P_2)=\frac{\partial}{\partial \overline{w}} P_1\frac{\partial}{\partial w}P_1\otimes P_2+P_1\otimes \frac{\partial}{\partial \overline{w}}P_2\frac{\partial}{\partial w}P_2={\mathcal K}(P_1)\otimes P_2+P_1\otimes {\mathcal K}(P_2).$$

Now assume that the conclusion holds for all $0 \leq i,j\leq k$, that is, $${\mathcal K}_{i,j}(P_1\otimes P_2)={\mathcal K}_{i,j}(P_1)\otimes P_2+P_1\otimes {\mathcal K}_{i,j}(P_2). $$ Then,
\begin{equation*}\begin{array}{lllll}
{\mathcal K}_{i+1,j}(P_1\otimes P_2)&=&(P_1\otimes P_2)\frac{\partial}{\partial w} ({\mathcal K}_{i,j}(P_1\otimes P_2))\\[4pt]
&=& (P_1\otimes P_2)\frac{\partial}{\partial w}({\mathcal K}_{i,j}(P_1)\otimes P_2+P_1\otimes {\mathcal K}_{i,j}(P_2))\\[4pt]
&=& P_1\frac{\partial}{\partial w} ({\mathcal K}_{i,j}(P_1))\otimes P_2+P_1{\mathcal K}_{i,j}(P_1)\otimes  P_2\frac{\partial}{\partial w}P_2 + P_1\frac{\partial}{\partial w} P_1\otimes P_2{\mathcal K}_{i,j}(P_2)+P_1\otimes P_2 \frac{\partial}{\partial w}( {\mathcal K}_{i,j}(P_2)).
\end{array}
\end{equation*}
Notice that since $P_2\frac{\partial}{\partial w}P_2=P_1\frac{\partial}{\partial w}P_1=0$, Definition \ref{holomorphic} gives
\begin{equation*}\begin{array}{llll}{\mathcal K}_{i+1,j}(P_1\otimes P_2)&=&P_1\frac{\partial}{\partial w}({\mathcal K}_{i,j}(P_1))\otimes P_2+P_1\otimes P_2 \frac{\partial}{\partial w}( {\mathcal K}_{i,j}(P_2))\\
&=&{\mathcal K}_{i+1,j}(P_1)\otimes P_2+P_1\otimes {\mathcal K}_{i+1,j}(P_2).
\end{array}
\end{equation*}
One shows in the same manner that $${\mathcal K}_{i,j+1}(P_1\otimes P_2)={\mathcal K}_{i,j+1}(P_1)\otimes P_2+P_1\otimes {\mathcal K}_{i,j+1}(P_2),$$
and therefore, the conclusion also holds in the case of $0  \leq i,j\leq k+1$.

\end{proof}

\begin{cor}\label{dcf} Let $\mathcal{E}_1$ and $\mathcal{E}_2$ be Hermitian holomorphic bundles over $\Omega$ of rank $n$ and $m$, respectively. For $i,j \geq 0$,
%If $\mathcal{E}_2$ is a line bundle, i.e. $dim \mathcal{E}_2(w)=1, w\in \Omega$.
$$\mathcal{K}_{\mathcal{E}_1\otimes \mathcal{E}_2, z^i{\bar z}^j}=\mathcal{K}_{\mathcal{E}_1, z^i{\bar z}^j}\otimes I_m+I_n\otimes \mathcal{K}_{\mathcal{E}_2,z^i{\bar z}^j}.$$

\end{cor}

\begin{proof} Let $P_1(w)$ and $P_2(w)$ be the orthogonal projections onto ${\mathcal E}_1$ and ${\mathcal E}_2$, respectively.   By Theorem \ref{cf}, we have
$ {\mathcal K}_{i,j}(P_1\otimes P_2)={\mathcal K}_{i,j}(P_1)\otimes P_2+P_1\otimes {\mathcal K}_{i,j}(P_2).$  Suppose that $${\mathcal E}_1(w)=\bigvee \limits_{s=1}^n e^1_s(w) \text{ and }{\mathcal E}_2(w)=\bigvee \limits_{t=1}^m e^2_t(w).$$ Then  $P_i(w)=\alpha_i(w)(\alpha_i^{*}(w)\alpha_i(w))^{-1}\alpha_i^{*}(w),$ where $$\alpha_1(w)(w_1,w_2,\cdots,w_n)=\sum\limits_{s=1}^nw_s e^1_s(w),$$
and $$\alpha_2(w)(w_1,w_2,\cdots,w_m)=\sum\limits_{t=1}^mw_te^2_t(w), $$ for all $w \in \Omega$ and for some $w_s, w_t\in {\mathbb C}$.
Now let $\{\sigma_i\}^n_{i=1}$ be an orthonormal basis for $\mathbb{C}^n$. Then for any $e^1_s(w)\otimes e^2_t(w)\in {\mathcal E}_1(w)\otimes {\mathcal E}_2(w),$  we have
  \begin{equation*}\begin{array}{llll}
  ({\mathcal K}_{i,j}(P_1)(w)\otimes  P_2(w))(e^1_s(w)\otimes e^2_t(w))&=&{\mathcal K}_{i,j}(P_1)(w)e^1_s(w)\otimes  e^2_t(w)\\
  &=&\alpha_1(w)(-\mathcal{K}_{{\mathcal E}_1,z^i{\overline
 z}^j}(w))h_1^{-1}(w)\alpha_1^*(w)e^1_s(w)\otimes  e^2_t(w)\\
 &=&\alpha_1(w)(-\mathcal{K}_{{\mathcal E}_1,z^i{\overline
 z}^j}(w))h_1^{-1}(w)\alpha_1^*(w)\alpha_1(w)(\sigma_s)\otimes e^2_t(w)\\
 &=&\alpha_1(w)(-\mathcal{K}_{{\mathcal E}_1,z^i{\overline
 z}^j}(w))(\sigma_s)\otimes e^2_t(w).
   \end{array}
  \end{equation*}
Similarly, we also have $$(P_1(w)\otimes  {\mathcal K}_{i,j}(P_2)(w))(e^1_s(w)\otimes e^2_t(w))=e^1_s(w)\otimes \alpha_2(w)(-\mathcal{K}_{{\mathcal E}_1,z^i{\overline
 z}^j}(w))(\sigma_t).$$ When ${\mathcal K}_{i,j}(P_1\otimes P_2)$ is viewed as a bundle map on $\mathcal{E}_1\otimes \mathcal{E}_2$, the corresponding matrix representation under the basis
 $\{e^1_s\otimes e^2_t: 1 \leq s \leq n, 1 \leq t \leq m\}$ is $\mathcal{K}_{\mathcal{E}_1\otimes \mathcal{E}_2, z^i{\bar z}^j}$. From the calculation above, we see that it can also be represented as $\mathcal{K}_{\mathcal{E}_1, z^i{\bar z}^j}\otimes I_m+I_n\otimes \mathcal{K}_{\mathcal{E}_2,z^i{\bar z}^j}$ and this finishes the proof.
\end{proof}

\begin{cor}Let $\mathcal{E}_1$ and $\mathcal{E}_2$ be as in Corollary \ref{dcf}.
If $\mathcal{E}_2$ is a line bundle, then $$\text{trace}\mathcal{K}_{\mathcal{E}_1\otimes \mathcal{E}_2, z^i{\bar z}^j}-\text{trace}\mathcal{K}_{\mathcal{E}_1, z^i{\bar z}^j}=\mathcal{K}_{\mathcal{E}_2,z^i{\bar z}^j}.$$

\end{cor}

By using Theorem \ref{Lin2} and Corollary \ref{dcf}, we arrive at the following main theorem of the section:

\begin{thm}Let $T, S\in B_n(\Omega)$ and let $T\sim_{u} (M^*_z, \mathcal{H}_{K})$.  Suppose that there exist an isometry $V$ and functions $\{e_1, e_2, \cdots, $ $e_m\} \subseteq \text{Hol}(\Omega)$ such that for every $0 \leq i,j \leq n$,
$$V\mathcal{K}_{\mathcal{E}_S,z^i{\bar z}^j}V^*-\mathcal{K}_{\mathcal{E}_T, z^i{\bar z}^j}=\frac{\partial^{i+j+2}}{\partial^{i+1} w \partial^{j+1} \overline{w}}\psi\otimes I_{n}, $$
where $\psi$ is the function with the property that $$\exp \psi(w)=\sum\limits_{i=1}^m |e_i(w)|^2.$$
 Then there exists an $M^*_z\otimes I_m$-invariant subspace $\mathcal{M}$ of $\mathcal{H}_{{K}}\otimes \mathbb{C}^m$ such that
 $$S\sim_{u} (M^*_z\otimes I_m)|_{\mathcal{M}}.$$
Moreover, when $\psi$ is bounded on $\Omega$, $S$ is similar to $T$.

\end{thm}

\end{document}